\documentclass[11pt,letterpaper]{amsart}
\usepackage[utf8]{inputenc}
\usepackage[english]{babel}
\usepackage{amsmath,amssymb,amsthm,mathrsfs,color,times,textcomp,yfonts,mathtools,cases}

\usepackage[top=3.5cm,bottom=3.5cm,left=3cm,right=3cm]{geometry}

\allowdisplaybreaks[4]

\usepackage{bm}
\usepackage[T1]{fontenc} %加粗

\usepackage{subcaption}
\usepackage[normalem]{ulem}
\usepackage[export]{adjustbox}
\usepackage{esint}
\usepackage{xcolor}
\usepackage{array}
\usepackage{booktabs} % 用于专业排版的表格线
\usepackage{siunitx}
\usepackage[colorlinks=true]{hyperref}
\hypersetup{urlcolor=blue, citecolor=blue, linkcolor=red}

\hypersetup{
	colorlinks=true,
	linkcolor=red
}

\usepackage[square,numbers]{natbib}
\usepackage{url}
\usepackage{xurl}
\usepackage{float}
\usepackage{ulem}
\usepackage{syntonly}
\usepackage{mathtools}
\usepackage{bm}
\usepackage{soul}
\usepackage{amsfonts,amsmath,latexsym,verbatim,amscd,mathrsfs,color,array}
\usepackage[colorlinks=true]{hyperref}
\usepackage{amsmath,amssymb,amsthm,amsfonts,graphicx,color}
\usepackage{amssymb}
\usepackage{epstopdf}

\newcommand{\R}{ \mathbb{R} }

\newtheorem{definition}{Definition}
\theoremstyle{plain}
\newtheorem{theorem}{Theorem}[section]
\newtheorem{lemma}[theorem]{Lemma}
\newtheorem{prop}[theorem]{Proposition}
\newtheorem{proposition}[theorem]{Proposition}
\newtheorem{corollary}[theorem]{Corollary}
\newtheorem{remark}{Remark}[theorem]
\newcommand{\bremark}{\begin{remark} \em}
	\newcommand{\eremark}{\end{remark} }

\numberwithin{equation}{section}

\title{Giga-Kohn-type results for the fully fractional heat equation}

\begin{document}

\author[Yannick Sire]{Yannick Sire}
\address{Yannick Sire, Department of Mathematics, Johns Hopkins University, 404 Krieger Hall, 3400 N. Charles Street, Baltimore, MD 21218, USA}
\email{ysire1@jhu.edu}

\author[Juncheng Wei]{Juncheng Wei}
\address{Juncheng Wei, Department of Mathematics, Chinese University of Hong
Kong, Shatin, NT, Hong Kong.}
\email{wei@math.cuhk.edu.hk}

\author[Ke Wu]{Ke Wu}
\address{Ke Wu, School of Mathematics, Yunnan Normal University, Kunming, 650500, China.}
\email{kewu@ynnu.edu.cn}

\author[Zikai Ye]{Zikai Ye}
\address{Zikai Ye, Department of Mathematics, Chinese University of Hong
Kong, Shatin, NT, Hong Kong.}
\email{zkye@math.cuhk.edu.hk}

\begin{abstract}
We consider the semilinear fully fractional heat equation
\begin{equation*}
    (\partial_t-\Delta)^\sigma u=|u|^{p-1}u \text{ in } \R^n\times\mathbb{R}_{-}, \quad 0<\sigma<1.
\end{equation*}
For $n\leq 2\sigma$ or $1<p\leq \frac{n+2\sigma}{n-2\sigma}$, we generalize the monotonicity formula and Liouville-type theorem when $\sigma=1$ proved by Giga and Kohn. In order to overcome the difficulty that this equation is nonlocal, we  give a new  interpretation of the classical Giga-Kohn's Pohozaev identity in terms of Hermite expansion. This insight is new and interesting even for $\sigma=1$. We further establish a space-time nonlocal monotonicity formula for the self-similar equation. As far as we are concerned, this is the first monotonicity formula for space-time nonlocal equations without using an extension by Stinga and Torrea.
\end{abstract}

\maketitle

\section{Introduction}\label{Sec:Introduction}

Finite time blow-up solutions of the Fujita equation
\begin{equation}\label{Fujita}
u_t=\Delta u+|u|^{p-1}u
\end{equation}
have been studied extensively in the last several years. 

% see \cite{Giga-Kohn1985, Giga-Kohn2,Giga-Matsui-Sasayama2004, Merle-Zaag1998, del20203D, del20204D,del20195D,Harada20205D, Harada20206D}.
The finite time blow-up of \eqref{Fujita} is said to be of Type I if 
\begin{equation*}
\limsup_{t\to T^-}(T-t)^{\frac{1}{p-1}}\|u(\cdot,t)\|_{L^{\infty}}<+\infty
\end{equation*}
and of Type II if 
\begin{equation*}
\limsup_{t\to T^-}(T-t)^{\frac{1}{p-1}}\|u(\cdot,t)\|_{L^{\infty}}  =+\infty.
\end{equation*}

When $n\leq 2$ or $1<p\leq \frac{n+2}{n-2}$, Giga-Kohn \cite{Giga-Kohn1985} classified Type I blow-up self-similar solutions and established a Liouville Theorem for Type I blow-up solutions of \eqref{Fujita} with the help of a monotonicity formula. They also showed that no nonnegative Type II blow-up solution exists in \cite{Giga-Kohn2}, while the nonnegative assumption was removed by Giga-Matsui-Sasayama in \cite{Giga-Matsui-Sasayama2004}. Type I blow-up solutions without similarity were classified by Giga-Kohn \cite{Giga-Kohn2} and Merle-Zaag \cite{Merle-Zaag1998} in the same range. There are also various existence results of Type II blow-up solutions when $n\leq 6$ and $p=\frac{n+2}{n-2}$, see \cite{del20203D,del20204D,del20195D,Harada20205D,Harada20206D}.

The purpose of this paper is to establish analogues of the Giga-Kohn-type classification and monotonicity results for the following fully fractional heat equation
\begin{equation}\label{fraceqn}
    (\partial_t-\Delta)^\sigma u=|u|^{p-1}u \text{ in } \R^n\times\mathbb{R}_{-}
\end{equation}
with $0<\sigma<1$, $p>1$ and $\mathbb{R}^{n}\times\mathbb{R}_{-}:=\mathbb{R}^{n}\times (-\infty, 0)$. Here, the fully fractional heat operator $(\partial_t-\Delta)^\sigma u$ is defined by the singular integral
\begin{equation}\label{point wise formula}
(\partial_t-\Delta)^\sigma u(x,t):=c_{n,\sigma}\int_{-\infty}^{t}\int_{\R^n}\frac{u(x,t)-u(y,T)}{(t-T)^{\frac{n}{2}+1+\sigma}}e^{-\frac{|x-y|^2}{4(t-T)}}dydT,
\end{equation}
where $c_{n,\sigma}:=\frac{1}{(4\pi)^{\frac{n}{2}}|\Gamma(-\sigma)|}$. For $u\in C^{2\sigma+\epsilon,\sigma+\epsilon}_{x,t,loc}(\mathbb{R}^n\times \mathbb{R}_{-})$ (see \cite[Definition 2.1]{Stinga-Torrea2017}) belonging to the following space of slowly increasing functions
\begin{equation*}
\mathcal{L}(\mathbb{R}^n\times\mathbb{R}_-):=\{u(x,t)\in L^1_{loc}(\mathbb{R}^n\times \mathbb{R}_-):\ \int_{-\infty}^{t}\int_{\mathbb{R}^n}\frac{|u(y,T)|e^{-\frac{|x-y|^2}{4(t-T)}}}{1+(t-T)^{\frac{n}{2}+1+\sigma}}dydT<+\infty,\ \forall\  t<0\},
\end{equation*}
the point-wise formula \eqref{point wise formula} is well-defined, see Stinga-Torrea \cite[Remark 2.2]{Stinga-Torrea2017}.

The fully fractional heat operator $(\partial_t-\Delta)^{\sigma}$ includes several classical nonlocal operators as special cases. When $u$ is independent of $t$, the operator $(\partial_t-\Delta)^{\sigma}$ reduces to the fractional Laplacian, $(-\Delta)^{\sigma}$, defined as
\begin{equation*}
(-\Delta)^{\sigma}u(x)=C_{n,\sigma}P.V.\int_{\mathbb{R}^n}\frac{u(x)-u(y)}{|x-y|^{n+2\sigma}}dy   
\end{equation*}
for some constant $C_{n,\sigma}$. When $u$ is independent of $x$, the operator $(\partial_t-\Delta)^{\sigma}$ reduces to $\partial_t^{\sigma}$, the Marchaud fractional derivative of order $\sigma$, defined as
\begin{equation*}
\partial_t^{\sigma}u(t)=C_\sigma\int_{-\infty}^{t}\frac{u(t)-u(T)}{(t-T)^{1+\sigma}}dT   
\end{equation*}
for some constant $C_{\sigma}$. Moreover, the fully fractional heat operator $(\partial_t-\Delta)^{\sigma}$ converges to the classical heat operator $\partial_t-\Delta$ as $\sigma\to 1$, see \cite{Wangqiuhong2023}.

The equation \eqref{fraceqn} has been studied in several recent works. In \cite{Ferreira2024}, Ferreira-de Pablo  considered the existence of the nonnegative memory data problem, Fujita-type exponent $p_{F,\sigma}=1+\frac{2\sigma}{n+2-2\sigma}$ and blow-up rates of \eqref{fraceqn}.

One prominent property of the equation \eqref{fraceqn} is that it is scale-invariant. That is, if $u(x, t)$ is a solution of \eqref{fraceqn}, then so is
$$u_{\lambda}(x, t)=\lambda^{2\beta}u(\lambda x, \lambda^{2}t),\qquad \beta=\frac{\sigma}{p-1}$$
for every $\lambda>0$.
\begin{definition}[Self-similar solutions]
If $u$ is a solution of \eqref{fraceqn} and $u_{\lambda}=u$ for every $\lambda>0$, then $u$ is called a self-similar solution of \eqref{fraceqn}.
\end{definition}

One can check that $u=0$ and $u=\pm\kappa (-t)^{-\frac{\sigma}{p-1}}\in C^{2\sigma+\epsilon,\sigma+\epsilon}_{x,t,loc}(\mathbb{R}^n\times \mathbb{R}_{-})\cap\mathcal{L}(\mathbb{R}^n\times\mathbb{R}_-)$ are self-similar solutions of \eqref{fraceqn}, where
$$\kappa=\left(\frac{\Gamma(\beta+\sigma)}{\Gamma(\beta)}\right)^{\frac{1}{p-1}}.$$ 

Our first main result shows that the only self-similar solutions are the trivial ones $u=0$ and $u=\pm\kappa (-t)^{-\frac{\sigma}{p-1}}$ provided that $n\leq 2\sigma$ or $1<p\leq \frac{n+2\sigma}{n-2\sigma}$ and $u$ satisfies a certain condition.
\begin{theorem}\label{main}
 For $n\leq 2\sigma$ or $1<p\leq \frac{n+2\sigma}{n-2\sigma}$, if $u$ is a self-similar weak solution of \eqref{fraceqn} satisfying 
\begin{equation}\label{type I assumption}
    \sup_{(x, t)\in\mathbb{R}^n\times\mathbb{R}_{-}}(-t)^{\beta}|u(x,t)|\leq M<\infty,
\end{equation}
then either $u=0$ or $u=\pm \kappa (-t)^{-\beta}$, where  $\beta=\frac{\sigma}{p-1}$ and $\kappa=\left(\frac{\Gamma(\beta+\sigma)}{\Gamma(\beta)}\right)^{\frac{1}{p-1}}$.
\end{theorem}    
\begin{remark}
The weak solution is defined in Definition \ref{Def:weaksol} below. From the condition of Theorem \ref{main}, one can check that automatically $u\in \mathcal{L}(\mathbb{R}^n\times \mathbb{R}_-)$. 
\end{remark}
\begin{remark}
Due to \cite[Theorem 2]{Quittner2024}, we know that if $\sigma=\frac{1}{2}, n=1$ or $1<p<\frac{n+1}{n-1}$ and $u$ is a positive smooth solution of \eqref{fraceqn}, then the condition \eqref{type I assumption} holds automatically.
\end{remark}
Our next main result is concerned with a classification of solutions without the hypothesis of self-similarity.
\begin{theorem}\label{Liouville}
 For $n\leq 2\sigma$ or $1<p\leq \frac{n+2\sigma}{n-2\sigma}$, if $u$ is a weak solution of \eqref{fraceqn} satisfying \eqref{type I assumption} and
\begin{equation*}
    \limsup_{t\to 0^-}(-t)^{\beta}|u(0,t)|>0,
\end{equation*}
then $u=\pm \kappa (-t)^{-\beta}$.
\end{theorem}

\textbf{Main difficulties and ideas:} We now explain the main difficulties and ideas in the proof. For the local case $\sigma=1$, Giga-Kohn studied the solutions to \eqref{Fujita} in self-similar variables
\begin{equation*}
v(z,s)=e^{\frac{s}{p-1}}u(e^{-\frac{s}{2}}z,-e^{-s}).
\end{equation*}
For a self-similar solution $u$, the function $v$ defined above solves
\begin{equation*}
(\mathcal{L}+\frac{1}{p-1})v=|v|^{p-1}v\text{ in }\mathbb{R}^n\times \mathbb{R},
\end{equation*}
where $\mathcal{L}:=-\Delta+\frac{1}{2}z\cdot \nabla$ is the Ornstein-Uhlenbeck operator. Their proof is mainly based on two ingredients. The first ingredient is a Pohozaev identity for self-similar solutions of \eqref{Fujita}:
\begin{equation}\label{Giga-Kohn Pohozaev identity}
\left(\frac{n}{p+1}+\frac{2-n}{2}\right)\int_{\mathbb{R}^n}\rho|\nabla v|^2+\frac{1}{2}\left(\frac{1}{2}-\frac{1}{p+1}\right)\int_{\mathbb{R}^n}\rho|z|^2|\nabla v|^2=0,
\end{equation}
where $\rho(z)=e^{-\frac{|z|^2}{4}}$. From this, they can deduce $v$ is constant. The second ingredient is a monotonicity formula
\begin{equation*}
E[v](a)-E[v](b)=\int_a^b\int_{\mathbb{R}^n}|v_s|^2\rho dzds
\end{equation*}
for the energy functional 
\begin{equation*}
E[v](s)=\frac{1}{2}\int_{\mathbb{R}^n}\rho|\nabla v|^2dz+\frac{1}{2(p-1)}\int_{\mathbb{R}^n}\rho v^2 dz-\frac{1}{p+1}\int_{\mathbb{R}^n}\rho |v|^{p+1}dz.
\end{equation*}
The proof of both facts relies on some cancellations which seem only to work for the Laplace operator. The results for the operator under consideration require new ideas we now explain.

For $0<\sigma<1$, we first derive a nonlocal elliptic equation (see \eqref{EmdenFowler} below) for self-similar solutions of \eqref{fraceqn}. We can also attempt to follow the procedure in \cite{Giga-Kohn1985} to establish a Pohozaev identity involving a series of complicated integrals. However, the kernels of those integrals are no longer homogeneous. Then many non-estimable remainder terms appear, and we can only obtain some partial results. Indeed, we can only obtain a Liouville theorem for $1<p\leq\frac{1+\sigma}{1-\sigma}+\epsilon$ when $n=1$ with $\sigma\in (0,\frac{1}{2})\cup (\frac{5}{6},1)$ and $1<p<\frac{n+2+2\sigma}{n+2-2\sigma}+\epsilon$ when $n\geq 2$ or $n=1$ with $\frac{1}{2}\leq \sigma\leq \frac{5}{6}$, for some small $\epsilon$. See \cite{SWWYRemark} for details.

Our starting point relies on the use of Hermite polynomials which diagonalizes the Ornstein-Uhlenbeck operator $\mathcal{L}$.  
Surprisingly, we observe that  Hermite polynomials $h_{\alpha}$'s are both eigenfunctions of $\mathcal{L}$ and its nonlocal counterpart $T_{\beta,\sigma}$ defined in \eqref{EmdenFowler} below. More precisely, $T_{\beta,\sigma}$ can be written spectrally as
\begin{equation*}
T_{\beta,\sigma}=\frac{\Gamma(\mathcal{L}+\beta+\sigma)}{\Gamma(\mathcal{L}+\beta)}
\end{equation*}
with $\Gamma(\cdot)$ denoting the Gamma function. See Lemma \ref{Teigen} for the rigorous statement. This fact is the cornerstone of our argument. 

With this observation, we can expand $v$ in terms of Hermite polynomials $h_{\alpha}$'s and establish a family of Pohozaev identities in terms of its coefficients of the expansion on the basis of Hermite polynomials. By cleverly adjusting some parameters in these identities, we can show that $v$ must restrict to the zero  mode, i.e. $v$ is constant. Even though we have assumed $0<\sigma<1$, the proof in this paper actually works for $\sigma=1$ and we provide a {\sl spectral } proof for Giga-Kohn's Pohozaev identity \eqref{Giga-Kohn Pohozaev identity}.
We would like to emphasize that our argument is somehow flexible enough to encapsulate similar computations if the operator $(\partial_t-\Delta)^\sigma$ is replaced by a more general master equation. Notice that, as alluded in \cite{Stinga-Torrea2017}, the operator $(\partial_t-\Delta)^\sigma$ belongs to the class of master equations in the Physics literature. 

In order to prove Theorem \ref{Liouville}, we will establish a monotonicity formula for a space-time nonlocal energy $\mathcal{E}[v]$ in Proposition \ref{monotonicity}. Once the monotonicity formula is established, we can  prove Theorem \ref{Liouville} by using an argument similar to the proof of Giga-Kohn's Liouville-type theorem \cite[Theorem 2']{Giga-Kohn1985}. We notice that Bustamante \cite{Bustamante2025} also proved a monotonicity formula for the Stinga-Torrea extension of $u$ by applying a high-dimensional parabolic-to-elliptic transformation. However, as far as we are concerned, this is the first space-time nonlocal monotonicity formula.

Finally, we would like to mention that another operator with longe-range diffusion can be considered, namely the operator $\partial_{t}+(-\Delta)^{\sigma}$. In this case, it is known that the operator {\sl cannot} be represented by an integral similar to \eqref{point wise formula}. In particular, the associated Ornstein-Uhlenbeck operator is not a power of the local one, and it is not clear which are the correct replacements for the Hermite polynomials. However, we mention that Deng-Sire-Wei-Wu \cite{Deng-Sire-Wei-Wu2021} proved  some Giga-Kohn-type results for the model
\begin{equation}\label{new model}
\partial_{t}u+(-\Delta)^{\frac{1}{2}}u-|u|^{p-1}u=0.
\end{equation}
with $1<p<p_*(n)$ for some $1<p_*(n)<\frac{n+1}{n-1}$. The half-Laplacian case has an explicit kernel and hence all the computations can be made explicitly. Due to some technical restrictions, the results in \cite{Deng-Sire-Wei-Wu2021} are not optimal. It will be an interesting question to know whether the ideas in this paper work for equation \eqref{new model}.

This paper will be organized as follows. In Section 2, we provide some preliminaries for the proof, including the regularity of the solution and the self-similar coordinates. In Section 3,  we establish a Giga-Kohn-type Pohozaev identity for the fully fractional heat operator in terms of Hermite coefficients, which can be used to finish the proof of Theorem \ref{main}. In Section 4, we derive a monotonicity formula for space-time nonlocal equations without using the Stinga-Torrea extension. In Section 5, we give the proof of  Theorem \ref{Liouville} by using the results in Section 3 and Section 4.

{\bf Notations:}
\begin{itemize}
    \item For a multi-index $\alpha=(\alpha_1,\dots,\alpha_n)\in \mathbb{N}^n$, $|\alpha|:=\alpha_1+\dots+\alpha_n$, $\alpha!:=\alpha_1!\dots \alpha_n!$.
\end{itemize}

\section{Preliminaries}\label{Sec:Preliminary}
In this section, we will introduce some preliminaries for our proof, including the regularity of the solution and the self-similar coordinates.

First, we introduce the concept of weak solutions of \eqref{fraceqn}. 
\begin{definition}[Weak solution]\label{Def:weaksol}
We say $u\in L^p_{loc}(\mathbb{R}^n\times \mathbb{R}_{-})\cap \mathcal{L}(\mathbb{R}^n\times\mathbb{R}_-)$ is a weak solution to \eqref{fraceqn} if \eqref{fraceqn} holds in the sense of distribution, i.e. 
\begin{equation*}
\int_{-\infty}^{0}\int_{\mathbb{R}^n}u(x,t)(-\partial_t-\Delta)^{\sigma}\varphi(x,t)dxdt=\int_{-\infty}^{0}\int_{\mathbb{R}^n}|u|^{p-1}u\varphi dxdt.
\end{equation*}
for any $\varphi\in C_c^{\infty}(\mathbb{R}^n\times \mathbb{R}_{-})$, where
\begin{equation*}
(-\partial_t-\Delta)^\sigma \varphi(x,t):=c_{n,\sigma}\int_{t}^{+\infty}\int_{\R^n}\frac{\varphi(x,t)-\varphi(y,T)}{(T-t)^{\frac{n}{2}+1+\sigma}}e^{-\frac{|x-y|^2}{4(T-t)}}dydT
\end{equation*}
is the adjoint operator of $(\partial_t-\Delta)^{\sigma}$.
\end{definition}

We start with the initial regularity of weak solutions.
\begin{proposition}\label{initial regularity}
Assume $u$ is a weak solution of \eqref{fraceqn} satisfying \eqref{type I assumption}.
Then there exists a small positive constant $\epsilon$ such that $u\in C^{2\sigma+\epsilon,\sigma+\epsilon}_{x,t,loc}(\mathbb{R}^n\times \mathbb{R}_{-})$.
\end{proposition}
\begin{proof}
Pick arbitrary parabolic cylinders $Q_0'\Subset Q_0\Subset Q_1\Subset Q_2\Subset \mathbb{R}^n\times \mathbb{R}_{-}$.
Take a cut-off function $\eta\in C_c^\infty(Q_2)$ with $0\leq \eta\leq 1$ and $\eta\equiv 1$ in a neighborhood of $Q_1$. Then we have $\eta u\in L^{\infty}(\mathbb{R}^{n+1})\cap L^2(\mathbb{R}^{n+1})$ and $\text{supp}(\eta u)\subset Q_2$.

For any $\varphi \in C_c^{\infty}(Q_1)$, since $u$ is a weak solution, we have
\begin{equation*}
\begin{aligned}
&\iint_{\mathbb{R}^n\times \mathbb{R}_-}\eta u(-\partial_t-\Delta)^{\sigma}\varphi\\
= &\iint_{\mathbb{R}^n\times \mathbb{R}_-}u(-\partial_t-\Delta)^{\sigma}\varphi+\iint_{\mathbb{R}^n\times \mathbb{R}_-}(\eta-1) u(-\partial_t-\Delta)^{\sigma}\varphi\\
=&\iint_{\mathbb{R}^n\times \mathbb{R}_-}|u|^{p-1}u\varphi+c_{n,\sigma}\iint_{Q_1^c}(\eta(x,t)-1) u(x,t)\int_{t}^{+\infty}\int_{\R^n}\frac{\varphi(x,t)-\varphi(y,T)}{(T-t)^{\frac{n}{2}+1+\sigma}}e^{-\frac{|x-y|^2}{4(T-t)}}dydTdxdt.
\end{aligned}
\end{equation*}
Since $\varphi(x,t)=0$ in $Q_1^c$, the second integral above can be further simplified to be
\begin{equation*}
\begin{aligned}
&c_{n,\sigma}\iint_{Q_1^c}(\eta(x,t)-1) u(x,t)\int_{t}^{+\infty}\int_{\R^n}\frac{\varphi(x,t)-\varphi(y,T)}{(T-t)^{\frac{n}{2}+1+\sigma}}e^{-\frac{|x-y|^2}{4(T-t)}}dydTdxdt\\
=&-c_{n,\sigma}\iint_{Q_1^c}(\eta(x,t)-1) u(x,t)\int_{t}^{+\infty}\int_{\R^n}\frac{\varphi(y,T)}{(T-t)^{\frac{n}{2}+1+\sigma}}e^{-\frac{|x-y|^2}{4(T-t)}}dydTdxdt\\
=&c_{n,\sigma}\iint_{\mathbb{R}^n\times\mathbb{R}_{-}}\varphi(y,T)\int_{-\infty}^{T}\int_{\R^n}\frac{1-\eta(x,t)}{(T-t)^{\frac{n}{2}+1+\sigma}}e^{-\frac{|x-y|^2}{4(T-t)}}u(x,t)dxdtdydT.
\end{aligned}
\end{equation*}
It follows that $\eta u$ is a weak solution of
\begin{equation}\label{etaueqn}
(\partial_t-\Delta)^{\sigma}(\eta u)=|u|^{p-1}u+f\qquad \text{in }Q_1,
\end{equation}
where
\begin{equation*}
f(x,t):=c_{n,\sigma}\int_{-\infty}^{t}\int_{\R^n}\frac{1-\eta(y,T)}{(t-T)^{\frac{n}{2}+1+\sigma}}e^{-\frac{|x-y|^2}{4(t-T)}}u(y,T)dydT.
\end{equation*}

\textbf{Claim:} $f\in L^{\infty}(Q_1)$.

Indeed, since $\eta\equiv 1$ in a neighborhood of $Q_1$, there exists a small $\delta>0$ such that for any $(x,t)\in Q_1$ and $(y,T)\in \mathbb{R}^{n+1}$ with
\begin{equation*}
|x-y|<\delta,\qquad 0<t-T<\delta^2,
\end{equation*}
we have $\eta(y,T)=1$. Let $t-T=\tau$ and $x-y=\xi$, then for any $(x,t)\in Q_1$,
\begin{equation*}
\begin{aligned}
|f(x,t)|
&\leq C\int_{-\infty}^{t-\delta^{2}}\int_{\mathbb{R}^n}\frac{|u(y,T)|}{(t-T)^{\frac{n}{2}+1+\sigma}}e^{-\frac{|x-y|^2}{4(t-T)}}dydT\\
&+C\int_{-\infty}^{t}\int_{|x-y|\geq\delta}\frac{|u(y,T)|}{(t-T)^{\frac{n}{2}+1+\sigma}}e^{-\frac{|x-y|^2}{4(t-T)}}dydT\\
&\leq C\int_{\delta^2}^{+\infty}\int_{\mathbb{R}^n}\tau^{-\frac{n}{2}-1-\sigma}e^{-\frac{|\xi|^2}{4\tau}}d\xi d\tau+C\int_{0}^{+\infty}\int_{|\xi|\geq \delta}\tau^{-\frac{n}{2}-1-\sigma}e^{-\frac{|\xi|^2}{4\tau}}d\xi d\tau\\
&\leq C\int_{\delta^2}^{+\infty}\tau^{-1-\sigma}d\tau+C\int_{0}^{+\infty}\tau^{-1-\sigma}e^{-\frac{c\delta^2}{\tau}}d\tau\leq C.
\end{aligned}
\end{equation*}
Similarly, we can show that $f\in C^{\infty}_{x,t}(Q_1)$.

Next, we take a mollifier $\rho_{\epsilon}\in C_c^{\infty}(\mathbb{R}^{n+1})$ with $
\text{supp}\rho_\epsilon\subset B_\epsilon\times (-\epsilon^2,\epsilon^2)$ and $\int_{\mathbb{R}^{n+1}}\rho_{\epsilon}=1$.
We choose an $\epsilon_0>0$ sufficiently small such that $B_{\epsilon_0}(Q_0):=\{(x,t)\in \mathbb{R}^{n+1}:\ |x-y|<\epsilon_0, |t-T|<\epsilon_0^2,\ \forall (y,T)\in Q_0\}\Subset Q_1$, then we have
\begin{equation*}
u_{\epsilon}:=\rho_\epsilon*(\eta u)\in C_c^{\infty}(\mathbb{R}^{n+1})
\end{equation*}
for any $0<\epsilon<\epsilon_0$. It follows that for any $\psi\in C_c^{\infty}(Q_0)$, we have
\begin{equation*}
\begin{aligned}
\iint_{\mathbb{R}^n\times \mathbb{R}_{-}}(\partial_t-\Delta)^{\sigma}u_{\epsilon}\psi
&=\iint_{\mathbb{R}^n\times \mathbb{R}_{-}}u_{\epsilon}(-\partial_t-\Delta)^{\sigma}\psi\\
&=\iint_{\mathbb{R}^n\times \mathbb{R}_{-}}\eta u\tilde \rho_{\epsilon}*((-\partial_t-\Delta)^{\sigma}\psi)\\
&=\iint_{\mathbb{R}^n\times \mathbb{R}_{-}}\eta u(-\partial_t-\Delta)^{\sigma}(\tilde \rho_{\epsilon}*\psi),
\end{aligned}
\end{equation*}
where $\tilde\rho_{\epsilon}(x,t):=\rho_{\epsilon}(-x,-t)$. Here in the second-to-last line, we have used the fact that $(-\partial_t-\Delta)^{\sigma}$ is translation invariant so that it commutes with convolution.

Since $\eta u$ is a weak solution to \eqref{etaueqn} in $Q_1$ and $\text{supp}(\tilde \rho_{\epsilon}*\psi)\subset Q_1$, we then have
\begin{equation*}
\begin{aligned}
\iint_{\mathbb{R}^n\times \mathbb{R}_{-}}\eta u(-\partial_t-\Delta)^{\sigma}(\tilde \rho_{\epsilon}*\psi)=&\iint_{\mathbb{R}^n\times \mathbb{R}_{-}}(|u|^{p-1}u+f)(\tilde \rho_{\epsilon}*\psi)\\
=&\iint_{\mathbb{R}^n\times \mathbb{R}_{-}}\rho_{\epsilon}*(|u|^{p-1}u+f)\cdot \psi.
\end{aligned}
\end{equation*}
That is,
\begin{equation*}
(\partial_t-\Delta)^{\sigma}u_{\epsilon}=\rho_{\epsilon}*(|u|^{p-1}u+f)\qquad a.e.\ \text{ in }Q_0.
\end{equation*}
With the local Schauder estimates \cite[Theorem 1.20]{Stinga-Torrea2017} of Stinga-Torrea, we have
\begin{equation*}
\|u_{\epsilon}\|_{C^{2\sigma+2\epsilon_1,\sigma+\epsilon_1}_{x,t}(Q_0')}\leq C(\|u\|_{L^{\infty}(Q_1)}+\|f\|_{C^{2\epsilon_1,\epsilon_1}_{x,t}(Q_1)})\leq C
\end{equation*}
for some constant $C>0$ and small $\epsilon_1>0$ independent of $\epsilon$. Taking $\epsilon\to 0$ and recalling that $\eta\equiv 1$ in $Q_0'$, we have
\begin{equation*}
u\in C^{2\sigma+2\epsilon_1,\sigma+\epsilon_1}_{x,t}(Q_0')
\end{equation*}
for any $Q_0'\Subset \mathbb{R}^{n}\times \mathbb{R}_{-}$. Since $Q_0'$ is arbitrary, the proof of the proposition is finished.
\end{proof}
As a consequence of Proposition \ref{initial regularity}, we know that if $u$ is a weak solution of \eqref{fraceqn} satisfying \eqref{type I assumption}, then the 
point-wise formula $(\partial_{t}-\Delta)^{\sigma}$ is well defined. The next proposition establishes the higher-order regularities of weak solutions.

\begin{proposition}\label{aprioriu}
Let $u$ be a weak solution to \eqref{fraceqn} satisfying \eqref{type I assumption}.
Then we have
\begin{equation}\label {prop apr 0}
\sup_{\mathbb{R}^n\times \mathbb{R}_{-}}(-t)^{\beta+\frac{1}{2}}|\nabla_{x}u(x,t)|\leq C.
\end{equation}
Furthermore, for any $\frac{1+\sigma}{2}<\theta<\min\{\frac{p}{2}+\sigma,1\}$, we have
\begin{equation}\label{prop apr 1}
\sup_{(x,t)\in\mathbb{R}^n\times \mathbb{R}_{-}}\sup_{|y|<\sqrt{-\frac{2}{3}t}}(-t)^{\beta+\theta} |\nabla u(x+y,t)-\nabla u(x,t)|\leq C|y|^{2\theta-1},    
\end{equation}
\begin{equation}\label{prop apr 2}
\sup_{(x,t)\in\mathbb{R}^n\times \mathbb{R}_{-}}\sup_{|t^\prime|< -\frac{1}{6}t}(-t)^{\beta+\theta}|u(x,t)-u(x,t+t^\prime)|\leq C|t^\prime|^{\theta},
\end{equation}
and
\begin{equation}\label{prop apr 3}
\sup_{(x,t)\in\mathbb{R}^n\times \mathbb{R}_{-}}\sup_{|t^\prime|< -\frac{1}{6}t}(-t)^{\beta+\frac{1}{2}+\theta}|\nabla_{x}u(x,t)-\nabla_{x}u(x,t+t^\prime)|\leq C|t^\prime|^{\theta},
\end{equation}
where $C$ is some constant depending only on $n$, $\sigma$, $p$, $\theta$ and $M$.
\end{proposition}
\begin{proof}
For any $x_0\in \mathbb{R}^n$ and $-\frac{7}{4}<t<-\frac{5}{4}$, we define
\begin{equation*}
\tilde u(y,\tau):=u(x_0+y,\tau+t).
\end{equation*}
Then we have
\begin{equation*}
(\partial_\tau-\Delta)^{\sigma}\tilde u=|\tilde u|^{p-1}\tilde u\quad \text{ in }B_2\times (-1,1)
\end{equation*}
and
\begin{equation*}
|\tilde u(y,\tau)|\leq M(-t-\tau)^{-\beta}\leq 4^{\beta}M
\end{equation*}
for any $\tau\in(-1, 1)$. The local Schauder estimates from Stinga-Torrea \cite[Theorem 1.20]{Stinga-Torrea2017} \footnote{The proof of the local Schauder estimate there actually only needs $\tilde u\in \mathcal{L}(\mathbb{R}^n\times \mathbb{R}_{-})$ and extending $\tilde u$ to be $0$ for $\tau\geq 1$. Since the estimates are interior estimates, we can get the same estimates for $\tilde u$.} yield that
\begin{equation*}
\|\tilde u\|_{C^{2\sigma,\sigma}_{x,t}(B_1\times (-\frac{1}{4},\frac{1}{4}))}\leq C    
\end{equation*}
when $\sigma\neq \frac{1}{2}$ and
\begin{equation*}
\|\tilde u\|_{\Lambda^{1,\frac{1}{2}}_{x,t}(B_1\times (-\frac{1}{4},\frac{1}{4}))}\leq C    
\end{equation*}
when $\sigma=\frac{1}{2}$, for some constant $C$ depending on $n$, $\sigma$, $p$ and $M$. Here $\Lambda^{1,\frac{1}{2}}_{x,t}(B_1\times (-\frac{1}{4},\frac{1}{4}))$ is the natural associated Zygmund space. 

We  choose $\theta$ so that $\frac{1+\sigma}{2}<\theta<\min\{\frac{p}{2}+\sigma,1\}$.
Iterating the local Schauder estimates (without loss of generality, we can take $B_2$ and the range of $t$ larger so that the final range of $x$ contains $B_1$ and the final range of $t$ contains $(-\frac{7}{4},-\frac{5}{4})$), we get that
\begin{equation*}
\|\tilde u\|_{C^{2\theta,\theta}_{x,t}(B_1\times (-\frac{1}{4},\frac{1}{4}))}\leq C    
\end{equation*}
where $C$ is some constant depending only on $n$, $\sigma$, $p$, $\theta$ and $M$. 
Since $x_0$ is chosen arbitrarily, we conclude that
\begin{equation*}
\sup_{\mathbb{R}^n\times (-\frac{7}{4},-\frac{5}{4})}|\nabla_{x}u(x,t)|\leq C.
\end{equation*}
Moreover, for any $x\in \mathbb{R}^n$, $|y|<1$ and $-\frac{7}{4}<t<-\frac{5}{4}$, 
\begin{equation*}
|\nabla_{x} u(x+y,t)-\nabla_{x} u(x,t)|\leq C|y|^{2\theta-1}.
\end{equation*}

Now for any $(x,t)\in \mathbb{R}^n\times \mathbb{R}_{-}$, we define $\lambda=\sqrt{-\frac{2}{3}t}$ and
\begin{equation*}
u_\lambda(y,\tau):=\lambda^{2\beta}u(x+\lambda y,\lambda^2\tau).
\end{equation*}
Then $u_\lambda$ satisfies the same equation and the same bound as $u$. Hence, the previous argument works for $u_\lambda$. In particular, we have
\begin{equation*}
|\nabla_{x}u(x,t)|=\lambda^{-2\beta-1}|\nabla_{y}u_\lambda(0,-\frac{3}{2})|\leq C(-t)^{-\beta-\frac{1}{2}}
\end{equation*}
and for any $|y|<\sqrt{-\frac{2}{3}t}$,
\begin{equation*}
|\nabla_{x} u(x+y,t)-\nabla_{x}u(x,t)|=\lambda^{-2\beta-1}|\nabla_{y}u_{\lambda}(\lambda^{-1}y,-\frac{3}{2})-\nabla_{y}u_\lambda(0,-\frac{3}{2})|\leq C(-t)^{-\beta-\theta}|y|^{2\theta-1}.    
\end{equation*}
We also have for any $|\tau|< \frac{1}{4}$,
\begin{equation*}
|u(x,t-\frac{2}{3}t\tau)-u(x,t)|=\lambda^{-2\beta}|u_{\lambda}(0,\tau-\frac{3}{2})-u_{\lambda}(0,-\frac{3}{2})|\leq C\lambda^{-2\beta}|\tau|^{\theta}.
\end{equation*}
Choosing $\tau=\frac{3}{2}(-t)^{-1}t^\prime$ with $|t^\prime|<-\frac{1}{6}t$, we have
\begin{equation*}
|u(x,t+t^\prime)-u(x,t)|\leq C(-t)^{-\beta-\theta}|t^\prime|^{\theta}.
\end{equation*}

The H\"older time-estimates for $\nabla_{x}u$ follow with a similar argument. Consequently, we obtain the desired result.
\end{proof}

Next, we introduce the following self-similar coordinates
\begin{equation}\label{self-similar coordinates}
    z=\frac{x}{\sqrt{-t}},\ s=-\log(-t),\ w=\frac{y}{\sqrt{-T}},\ S=-\log(-T)
\end{equation}
and write
\begin{equation}\label{self-similar transform}
    u(x,t)=(-t)^{-\beta}v(z,s).
\end{equation}

Then we can rewrite Proposition \ref{aprioriu} into self-similar coordinates.
\begin{prop}\label{aprioriv}
Assume $u$ is a weak solution to \eqref{fraceqn}  and $v(z,s)=(-t)^{\beta}u(x,t)$ is uniformly bounded, then we have
\begin{equation*}
\sup_{\mathbb{R}^n\times \mathbb{R}}|\nabla_{z}v(z,s)|\leq C.
\end{equation*}
Furthermore, for any $\frac{1+\sigma}{2}<\theta<\min\{\frac{p}{2}+\sigma,1\}$, we have
\begin{equation*}
\sup_{(z,s)\in\mathbb{R}^n\times \mathbb{R}}\sup_{|z-w|\leq \frac{2}{3}}|\nabla_{z}v(z,s)-\nabla_{z}v(w,s)|\leq C|z-w|^{2\theta-1},
\end{equation*}

\begin{equation*}
\sup_{(z,s)\in\mathbb{R}^n\times \mathbb{R}}\sup_{|\tau|\leq \frac{1}{8}}|v(z,s)-v(z,s+\tau)|\leq C|\tau|^{\theta}(1+|z|),
\end{equation*}
and
\begin{equation*}
\sup_{(z,s)\in\mathbb{R}^n\times \mathbb{R}}\sup_{|\tau|\leq \frac{1}{8}}|\nabla_{z} v(z,s)-\nabla_{z} v(z,s+\tau)|\leq C |\tau|^{\theta}(1+|z|),
\end{equation*}
where $C$ is some constant depending only on $n$, $\sigma$, $p$, $\theta$ and $M$.
\end{prop}
\begin{proof}
The gradient estimates follow directly from Proposition \ref{aprioriu}. In the following, we will prove the H\"older estimates.

For any $z,w\in \mathbb{R}^n$ with $|z-w|\leq \frac{2}{3}$, we have
\begin{equation*}
   |e^{-\frac{s}{2}}z-e^{-\frac{s}{2}}w|\leq \frac{2}{3}e^{-\frac{s}{2}}<\sqrt{-\frac{2}{3}t}.
\end{equation*}
From \eqref{prop apr 1}, we know that
\begin{equation*}
\sup_{(x,t)\in\mathbb{R}^n\times \mathbb{R}_{-}}\sup_{|y|<\sqrt{-\frac{2}{3}t}}(-t)^{\beta+\theta} |\nabla_{x}u(x+y,t)-\nabla_{x}u(x,t)|\leq C |y|^{2\theta-1}.    
\end{equation*}
where $C$ is some constant depending only on $n$, $\sigma$, $p$, $\theta$ and $M$. Thus
\begin{equation*}
\begin{aligned}
|\nabla _z v(z,s)-\nabla_z v(w,s)|
&=e^{-(\beta+\frac{1}{2})s}|\nabla_x u(e^{-\frac{s}{2}}z,-e^{-s})-\nabla_x u(e^{-\frac{s}{2}}w,-e^{-s})|\\
&\leq Ce^{-(\beta+\frac{1}{2})s}e^{(\beta+\theta)s}\left|e^{-\frac{s}{2}}(z-w)\right|^{2\theta-1}\\
&=C|z-w|^{2\theta-1}.
\end{aligned}
\end{equation*}

For any $z\in \mathbb{R}^n$, $s\in \mathbb{R}$, $|\tau|\leq \frac{1}{8}$, we have
\begin{equation*}
   |e^{-s-\tau}-e^{-s}|<\frac{1}{6}e^{-s}.
\end{equation*}
From \eqref{prop apr 2}, we deduce that
\begin{equation*}
\sup_{(x,t)\in\mathbb{R}^n\times \mathbb{R}_{-}}\sup_{|t^\prime|< -\frac{1}{6}t}(-t)^{\beta+\theta}|u(x,t)-u(x,t+t^\prime)|\leq C|t^\prime|^{\theta}.
\end{equation*}
Hence
\begin{equation*}
\begin{aligned}
|v(z,s)-v(z,s+\tau)|
&=\left|e^{-s\beta}u(e^{-\frac{s}{2}}z,-e^{-s})-e^{-(s+\tau)\beta}u(e^{-\frac{s+\tau}{2}}z,-e^{-s-\tau})\right|\\
&\leq e^{-s\beta}\left|u(e^{-\frac{s}{2}}z,-e^{-s})-u(e^{-\frac{s+\tau}{2}}z,-e^{-s})\right|\\
&+ e^{-s\beta}\left|u(e^{-\frac{s+\tau}{2}}z,-e^{-s})-u(e^{-\frac{s+\tau}{2}}z,-e^{-s-\tau})\right|\\
&+e^{-s\beta}\left|1-e^{-\tau\beta}\right|\cdot \left|u(e^{-\frac{s+\tau}{2}}z,-e^{-s-\tau})\right|\\
&\leq Ce^{-s\beta}e^{(\beta+\frac{1}{2})s}\cdot e^{-\frac{s}{2}}\left|1-e^{-\frac{\tau}{2}}\right|\cdot |z|\\
&+Ce^{-s\beta}e^{(\beta+\theta)s}e^{-s\theta}\left|1-e^{-\tau}\right|^{\theta}\\
&+Ce^{-s\beta}\left|1-e^{-\tau\beta}\right|e^{s\beta}\\
&\leq C(1+|z|)|\tau|^{\theta}.
\end{aligned}
\end{equation*}

 The H\"older estimates in time for $\nabla_{z} v$ follow with a similar argument, hence the proof is finished.
\end{proof}
The next result is concerned with how to express the nonlocal operator $(\partial_t-\Delta)^\sigma u(x,t)$ in the self-similar coordinates. The proof here is inspired by a computation used in \cite[Formula 1.6]{Wei2017}.
\begin{lemma}\label{FracOperatorSelfSim}
Let $\epsilon>0$ and $u\in C^{2\sigma+\epsilon,\sigma+\epsilon}_{x,t,loc}(\mathbb{R}^n\times \mathbb{R}_{-})\cap\mathcal{L}(\mathbb{R}^n\times\mathbb{R}_-)$ be a function satisfying \eqref{type I assumption}. Assume  $v(z,s)$ is defined via \eqref{self-similar coordinates} and \eqref{self-similar transform}, then we have
\begin{equation*}
\begin{aligned}
(\partial_t-\Delta)^\sigma u(x,t)=&e^{(\beta+\sigma)s}
\left[\int_{-\infty}^{0}\int_{\R^n}(v(z,s)-v(w,s+\tau))M_\tau(z,w)dwd\tau+Av(z,s)\right].
\end{aligned}
\end{equation*}
where $\tau=S-s$, $A=\frac{\Gamma(\beta+\sigma)}{\Gamma(\beta)}$ and the kernel $M_\tau (z,w)$ is given by
\begin{equation*}
   M_\tau(z,w)= c_{n,\sigma}\frac{e^{(\beta-\frac{n}{2}-1)\tau}}{(e^{-\tau}-1)^{\frac{n}{2}+1+\sigma}}e^{-\frac{|z-e^{-\frac{\tau}{2}}w|^2}{4(e^{-\tau}-1)}}.
\end{equation*}
Furthermore, if $u$ is self-similar, i.e. $v$ is independent of $s$, then
\begin{equation*}
\begin{aligned}
(\partial_t-\Delta)^\sigma u(x,t)=e^{(\beta+\sigma)s}
&\left[e^{\frac{|z|^2}{4}}\int_{\R^n}(v(z)-v(w))K(z,w)dw+Av(z)\right],
\end{aligned}
\end{equation*}
where the symmetric kernel $K(z,w)$ is given by
\begin{equation*}
\begin{aligned}
K(z,w)=&e^{-\frac{|z|^{2}}{4}}\int_{-\infty}^{0}M_{\tau}(z, w)d\tau\\
=&c_{n,\sigma}\int_{-\infty}^{0}\frac{e^{(\beta-\frac{n}{2}-1)\tau}}{(e^{-\tau}-1)^{\frac{n}{2}+1+\sigma}}\exp\left(-\frac{e^{-\tau}|z-w|^2}{4(e^{-\tau}-1)}-\frac{e^{-\frac{\tau}{2}}}{2(e^{-\frac{\tau}{2}}+1)}\langle z,w \rangle \right)d\tau.
\end{aligned}
\end{equation*}
\end{lemma}
\begin{proof}
By the assumptions, the  nonlocal operator $(\partial_t-\Delta)^\sigma u(x,t)$ is well defined. Moreover,
\begin{equation*}
\begin{aligned}
&(\partial_t-\Delta)^\sigma u(x,t)\\
=&c_{n,\sigma}\int_{-\infty}^{t}\int_{\R^n}\frac{u(x,t)-u(y,T)}{(t-T)^{\frac{n}{2}+1+\sigma}}e^{-\frac{|x-y|^2}{4(t-T)}}dydT\\
=&c_{n,\sigma}\int_{-\infty}^{t}\int_{\R^n}\frac{(-t)^{-\beta}v(\frac{x}{\sqrt{-t}}, -\log(-t))-(-T)^{-\beta}v(\frac{y}{\sqrt{-T}}, -\log(-T))}{(t-T)^{\frac{n}{2}+1+\sigma}}e^{-\frac{|x-y|^2}{4(t-T)}}dydT\\
=&c_{n,\sigma}\int_{-\infty}^{s}\int_{\R^n}\frac{e^{\beta s}v(z, s)-e^{\beta S}v(w, S)}{(e^{-S}-e^{-s})^{\frac{n}{2}+1+\sigma}}e^{-\frac{|e^{-\frac{s}{2}}z-e^{-\frac{S}{2}}w|^2}{4(e^{-S}-e^{-s})}}e^{-\frac{n+2}{2}S}dwdS.
\end{aligned}
\end{equation*}
Let $S=s+\tau$, then
\begin{equation*}
\begin{aligned}
&c_{n,\sigma}\int_{-\infty}^{s}\int_{\R^n}\frac{e^{\beta s}v(z, s)-e^{\beta S}v(w, S)}{(e^{-S}-e^{-s})^{\frac{n}{2}+1+\sigma}}e^{-\frac{|e^{-\frac{s}{2}}z-e^{-\frac{S}{2}}w|^2}{4(e^{-S}-e^{-s})}}e^{-\frac{n+2}{2}S}dwdS\\
=&c_{n,\sigma}\int_{-\infty}^{0}\int_{\R^n}\frac{e^{\beta s}v(z, s)-e^{\beta (s+\tau)}v(w, s+\tau)}{(e^{-(s+\tau)}-e^{-s})^{\frac{n}{2}+1+\sigma}}e^{-\frac{|e^{-\frac{s}{2}}z-e^{-\frac{s+\tau}{2}}w|^2}{4(e^{-(s+\tau)}-e^{-s})}}e^{-\frac{n+2}{2}(s+\tau)}dwd\tau\\
=&c_{n,\sigma}e^{(\beta+\sigma)s}\int_{-\infty}^{0}\int_{\R^n}\frac{v(z, s)-e^{\beta\tau}v(w, s+\tau)}{(e^{-\tau}-1)^{\frac{n}{2}+1+\sigma}}e^{-\frac{|z-e^{-\frac{\tau}{2}}w|^2}{4(e^{-\tau}-1)}}e^{-\frac{n+2}{2}\tau}dwd\tau\\
=&c_{n,\sigma}e^{(\beta+\sigma)s}\int_{-\infty}^{0}\int_{\R^n}\frac{v(z, s)-v(w, s+\tau)}{(e^{-\tau}-1)^{\frac{n}{2}+1+\sigma}}e^{-\frac{|z-e^{-\frac{\tau}{2}}w|^2}{4(e^{-\tau}-1)}}e^{(\beta-\frac{n+2}{2})\tau}dwd\tau\\
&+c_{n,\sigma}e^{(\beta+\sigma)s}\left[\int_{-\infty}^{0}\int_{\R^n}\frac{1-e^{\beta\tau}}{(e^{-\tau}-1)^{\frac{n}{2}+1+\sigma}}e^{-\frac{|z-e^{-\frac{\tau}{2}}w|^2}{4(e^{-\tau}-1)}}e^{-\frac{n+2}{2}\tau}dwd\tau\right] v(z, s).
\end{aligned}
\end{equation*}
Notice that
\begin{equation}\label{IntegralA}
\begin{aligned}
&c_{n,\sigma}\int_{-\infty}^{0}\int_{\R^n}\frac{(1-e^{\beta\tau})e^{-(\frac{n}{2}+1)\tau}}{(e^{-\tau}-1)^{\frac{n}{2}+1+\sigma}}e^{-\frac{|z-e^{-\frac{\tau}{2}}w|^2}{4(e^{-\tau}-1)}}dwd\tau\\
=&c_{n,\sigma}\int_{-\infty}^{0}\int_{\R^n}e^{-\frac{e^{-\tau}|w|^2}{4(e^{-\tau}-1)}}dw\frac{(1-e^{\beta\tau})e^{-(\frac{n}{2}+1)\tau}}{(e^{-\tau}-1)^{\frac{n}{2}+1+\sigma}}d\tau\\
=&c_{n,\sigma}(4\pi)^{\frac{n}{2}}\int_{-\infty}^{0}\frac{(1-e^{\beta\tau})e^{-\tau}}{(e^{-\tau}-1)^{1+\sigma}}d\tau=\frac{\Gamma(\beta+\sigma)}{\Gamma(\beta)}=A.    
\end{aligned}
\end{equation}
We complete the proof of the lemma.
\end{proof}
As a consequence of Proposition \ref{initial regularity} and Lemma \ref{FracOperatorSelfSim}, we have the following result.
\begin{corollary}
Let $u$ be a weak solution to \eqref{fraceqn} satisfying \eqref{type I assumption}. If $v$ is given by \eqref{self-similar coordinates} and \eqref{self-similar transform}, then $v$ satisfies the equation
\begin{equation}\label{EmdenFowler}
\int_{-\infty}^{0}\int_{\R^n}(v(z,s)-v(w,s+\tau))M_\tau(z,w)dwd\tau+Av(z,s)=|v(z,s)|^{p-1}v(z,s).    \end{equation}
If we assume further that $u$ is self-similar, i.e. $v$ is independent of $s$, then we have
\begin{equation}\label{selfsimeqn}
T_{\beta,\sigma}v(z)=e^{\frac{|z|^2}{4}}\int_{\R^n}(v(z)-v(w))K(z,w)dw+Av(z)=|v(z)|^{p-1}v(z).
\end{equation}
\end{corollary}
The kernel $M_\tau(z,w)$ can be written as 
\begin{equation*}
    M_\tau(z,w)=m(\tau)p_\tau(z,w),
\end{equation*}
where
\begin{equation}\label{nonlocal kernel}
    m(\tau)=\frac{e^{(\beta+\sigma)\tau}}{|\Gamma(-\sigma)|(1-e^\tau)^{1+\sigma}},\quad p_\tau(z,w)=\frac{1}{(4\pi (1-e^\tau))^{\frac{n}{2}}}e^{-\frac{|w-e^{\frac{\tau}{2}}z|^2}{4(1-e^\tau)}}.
\end{equation}
The function $p_\tau(z,w)$ satisfies the following Chapman–Kolmogorov equation. 
\begin{lemma}\label{CK}
For each fixed $\tau<0$, we have for any $z,w\in \R^n$,
\begin{equation*}
\rho(z)p_\tau(z,w)=\int_{\R^n}\rho(\eta)p_{\tau/2}(\eta,z)p_{\tau/2}(\eta,w)d\eta,
\end{equation*}
where $\rho(z)=e^{-\frac{|z|^2}{4}}$.
\end{lemma}
\begin{proof}
The proof of this result is elementary, we will defer it to Appendix \ref{Appendix A}.
\end{proof}

As mentioned in the introduction, in order to overcome the difficulty that  $T_{\beta,\sigma}$ is a nonlocal operator, we need to give a new interpretation of \eqref{Giga-Kohn Pohozaev identity} and provide a {\sl nonlocal} version of it. In the following, we will study the spectrum of the nonlocal linear operator $T_{\beta,\sigma}$. 

Let $\{h_{\alpha}\}_{\alpha\in \mathbb{N}^n}$ be the $n$-dimensional normalized Hermite polynomials given by
\begin{equation*}
h_{\alpha}(z)=\prod_{j=1}^{n}h_{\alpha_j}(z_j),
\end{equation*}
where $\alpha=(\alpha_1,\dots,\alpha_n)\in \mathbb{N}^n$ is the multi-index and $h_k(x)$ is the $1$-dimensional normalized Hermite polynomial given by
\begin{equation*}
h_{k}(x)=\frac{(-1)^k2^{\frac{k}{2}}}{\sqrt{k!}}e^{\frac{x^2}{4}}\frac{d^k}{dx^k}(e^{-\frac{x^2}{4}}).
\end{equation*}
Recall that $h_{\alpha}$'s are eigenfunctions of the Ornstein-Uhlenbeck operator $\mathcal{L}:=-\Delta+\frac{z}{2}\cdot \nabla$  with
\begin{equation*}
    \mathcal{L}h_\alpha=\frac{|\alpha|}{2}h_{\alpha}.
\end{equation*}
Furthermore, $\{h_{\alpha}\}_{\alpha\in \mathbb{N}^n}$ forms an orthonormal basis of the Hilbert space
\begin{equation*}
L^2_\rho(\mathbb{R}^n):=\{f\in L^2_{loc}(\mathbb{R}^n):\ \int_{\mathbb{R}^n}|f(z)|^2\rho(z)dz<\infty\}
\end{equation*}
with the associated inner product
\begin{equation*}
\langle f,g\rangle_\rho:=(4\pi)^{-\frac{n}{2}}\int_{\mathbb{R}^n}f(z)g(z)\rho(z)dz.
\end{equation*}
Moreover, we have
\begin{equation*}
\langle h_{\alpha},h_{\beta}\rangle_\rho=\delta_{\alpha\beta}.
\end{equation*}

In order to investigate  the spectrum of the nonlocal  linear operator $T_{\beta,\sigma}$, we introduce the space
$$H^{1}_{\sigma, \rho}(\mathbb{R}^{n}):=\{f\in L_{loc}^{2}(\mathbb{R}^{n}):\int_{\mathbb{R}^{n}}\int_{\mathbb{R}^{n}}(f(z)-f(w))^{2}K(z, w) dzdw+\int_{\mathbb{R}^{n}}f^{2}(z)\rho (z)dz<\infty\},$$
then $H^{1}_{\sigma, \rho}(\mathbb{R}^{n})$ is a Hilbert space with respect to the inner product
$$\langle f, g\rangle_{\sigma, \rho}=\int_{\mathbb{R}^{n}}\int_{\mathbb{R}^{n}}(f(z)-f(w))(g(z)-g(w))K(z, w) dzdw+\langle f,g\rangle_\rho.$$
One can check that the embedding $H_{\sigma,\rho}^{1}(\mathbb{R}^{n})\to L^2_\rho(\mathbb{R}^{n})$ is compact. Hence, standard spectral theory yields that $T_{\beta, \sigma}$ has a sequence of discrete eigenvalues.

The next lemma yields that $h_{\alpha}$'s are also eigenfunctions of $T_{\beta,\sigma}$ in $H^{1}_{\sigma, \rho}(\mathbb{R}^{n})$. This property is the main observation we used in the proof of Theorem \ref{main}.
\begin{lemma}\label{Teigen}
For any $\alpha\in \mathbb{N}^n$, we have $h_{\alpha}\in H^{1}_{\sigma, \rho}(\mathbb{R}^{n})$ and
\begin{equation*}
T_{\beta,\sigma}h_{\alpha}=t_{|\alpha|}h_{\alpha},
\end{equation*}
where $t_{|\alpha|}:=\frac{\Gamma(\beta+\sigma+\frac{|\alpha|}{2})}{\Gamma(\beta+\frac{|\alpha|}{2})}$.
\end{lemma}
\begin{proof}
We will provide a detailed proof in Appendix \ref{Appendix A}.
\end{proof}

\section{Classification of self-similar solutions}\label{Sec:Pohozaev}
In this section, we establish the Giga-Kohn Pohozaev identity for the fully fractional heat operator in terms of Hermite coefficients.

By our assumption, $v$ is uniformly bounded and measurable. It follows that $v\in L^2_\rho(\mathbb{R}^n)$ so that it can be expanded in terms of Hermite polynomials as $v=\sum_{\alpha\in \mathbb{N}^n}c_{\alpha}h_{\alpha}$. For the coefficients $c_{\alpha}$, we have the following Pohozaev identity.
\begin{prop}\label{GK1}
Suppose $u$ is a self-similar weak solution to \eqref{fraceqn}  and $v(z)=(-t)^{\beta}u((-t)^{\frac{1}{2}}z)$ is uniformly bounded. Write $v=\sum_{\alpha\in \mathbb{N}^n}c_{\alpha}h_{\alpha}$, then for any $n\geq 1$ and $p>1$, we have
\begin{equation}
\begin{aligned}
2\sum_{\alpha\in \mathbb{N}^n}|\alpha|t_{|\alpha|}c_{\alpha}^2+\sum_{\alpha\in \mathbb{N}^n}\sum_{j=1}^{n}\frac{|\alpha|t_{|\alpha|}}{\beta+\frac{|\alpha|}{2}}b_{\alpha,j}c_{\alpha}c_{\alpha+2e_j}=0,
\end{aligned}
\end{equation}
where $b_{\alpha,j}:=\sqrt{(\alpha_j+1)(\alpha_j+2)}$.
\end{prop}
\begin{proof}
The recurrence formula (see \cite{AndrewsAskeyRoy1999}) of Hermite polynomials yields that for any $\alpha\in \mathbb{N}^n$, $j=1,\dots,n$, we have
\begin{equation*}
\partial_j h_{\alpha}=\sqrt{\frac{\alpha_j}{2}}h_{\alpha-e_j}
\end{equation*}
and
\begin{equation*}
z_j h_{\alpha}(z)=\sqrt{2(\alpha_j+1)}h_{\alpha+e_j}+\sqrt{2\alpha_j}h_{\alpha-e_j},
\end{equation*}
where $e_j$ is the unit vector in $\mathbb{R}^n$ whose $j$-th coordinate is $1$ and $0$ elsewhere. It follows that
\begin{equation}\label{z2-2n}
(|z|^2-2n)h_{\alpha}=4|\alpha|h_{\alpha}+2\sum_{j=1}^{n}b_{\alpha,j}h_{\alpha+2e_j}+2\sum_{j=1}^{n}b_{\alpha-2e_j,j}h_{\alpha-2e_j}
\end{equation}
and
\begin{equation}\label{znabla}
z\cdot \nabla h_{\alpha}=|\alpha|h_{\alpha}+\sum_{j=1}^{n}b_{\alpha-2e_j,j}h_{\alpha-2e_j}.
\end{equation}

By \eqref{selfsimeqn}, we see that $T_{\beta,\sigma}v\in L^{\infty}(\mathbb{R}^n)\subset L^2_\rho(\mathbb{R}^n)$. With Lemma \ref{Teigen}, we have
\begin{equation}\label{Tbetasigmaexp}
T_{\beta,\sigma}v=\sum_{\alpha\in \mathbb{N}^n}t_{|\alpha|}c_{\alpha}h_{\alpha}.
\end{equation}
Testing \eqref{selfsimeqn} against $(|z|^2-2n)\rho v$ and $\rho z\cdot \nabla v$ respectively, we can get
\begin{equation}\label{Test1}
\begin{aligned}
0=&\ \langle T_{\beta,\sigma}v, (|z|^2-2n)v\rangle _{\rho}-(4\pi)^{-\frac{n}{2}}\int_{\mathbb{R}^n}\rho (|z|^2-2n)|v|^{p+1}\\
=&\ 4\sum_{\alpha\in \mathbb{N}^n}|\alpha|t_{|\alpha|}c_{\alpha}^2+2\sum_{\alpha\in \mathbb{N}^n}\sum_{j=1}^{n}(t_{|\alpha|}+t_{|\alpha|+2})b_{\alpha,j}c_{\alpha}c_{\alpha+2e_j}\\
&-(4\pi)^{-\frac{n}{2}}\int_{\mathbb{R}^n}\rho (|z|^2-2n)|v|^{p+1}
\end{aligned}
\end{equation}
and
\begin{equation}\label{Test2}
\begin{aligned}
0=&\ \langle T_{\beta,\sigma}v, z\cdot \nabla v\rangle _{\rho}-(4\pi)^{-\frac{n}{2}}\int_{\mathbb{R}^n}\rho |v|^{p-1}z\cdot \nabla v\\
=&\ \sum_{\alpha\in \mathbb{N}^n}|\alpha|t_{|\alpha|}c_{\alpha}^2+\sum_{\alpha\in \mathbb{N}^n}\sum_{j=1}^{n}t_{|\alpha|}b_{\alpha,j}c_{\alpha}c_{\alpha+2e_j}-\frac{(4\pi)^{-\frac{n}{2}}}{2(p+1)}\int_{\mathbb{R}^n}\rho (|z|^2-2n)|v|^{p+1}.
\end{aligned}
\end{equation}
Combining $-\eqref{Test1}+2(p+1)\eqref{Test2}$\footnote{If $\sigma=1$, this is exactly the Pohozaev identity obtained in \cite[Formula 3.11]{Giga-Kohn1985}.}, we obtain
\begin{equation}\label{classical Pohozaev}
0=2(p-1)\sum_{\alpha\in \mathbb{N}^n}|\alpha|t_{|\alpha|}c_{\alpha}^2+2\sum_{\alpha\in \mathbb{N}^n}\sum_{j=1}^{n}(pt_{|\alpha|}-t_{|\alpha|+2})b_{\alpha,j}c_{\alpha}c_{\alpha+2e_j}.
\end{equation}
Recall that $t_k=\frac{\Gamma(\beta+\sigma+\frac{k}{2})}{\Gamma(\beta+\frac{k}{2})}$ and $\beta=\frac{\sigma}{p-1}$, then we have
\begin{equation*}
pt_{|\alpha|}-t_{|\alpha|+2}=\bigg(p-\frac{\beta+\sigma+\frac{|\alpha|}{2}}{\beta+\frac{|\alpha|}{2}}\bigg)t_{|\alpha|}=\frac{p-1}{2}\frac{|\alpha|}{\beta+\frac{|\alpha|}{2}}t_{|\alpha|}
\end{equation*}
and the desired result follows.
\end{proof}

With the help of the Pohozaev identity above, we are able to prove Theorem \ref{main}.
\begin{proof}[Proof of Theorem \ref{main}]
It suffices to prove that: suppose $u\in \mathcal{L}(\mathbb{R}^n\times\mathbb{R}_-)$ is a self-similar solution to \eqref{fraceqn} and $v(z)=(-t)^{\beta}u((-t)^{\frac{1}{2}}z)$ is uniformly bounded, then $v$ must be a constant.    

We expand $v=\sum_{\alpha\in \mathbb{N}^n}c_\alpha h_{\alpha}$. In the following, we will prove $c_{\alpha}=0$ for all $|\alpha|\geq 1$.

From Proposition \ref{GK1} and recall that $b_{\alpha,j}=\sqrt{(\alpha_j+1)(\alpha_j+2)}$, we have
\begin{equation}\label{Poho}
\begin{aligned}
2\sum_{\alpha\in \mathbb{N}^n}|\alpha|t_{|\alpha|}c_{\alpha}^2+\sum_{k=1}^{\infty}\sum_{\alpha\in \mathbb{N}^n,|\alpha|=k}\sum_{j=1}^{n}\frac{kt_k}{\beta+\frac{k}{2}}\sqrt{(\alpha_j+1)(\alpha_j+2)}c_{\alpha}c_{\alpha+2e_j}=0.
\end{aligned}
\end{equation}

For each $k\geq 1$, we apply Young's inequality to get for any $\alpha\in \mathbb{N}^n$ with $|\alpha|=k$, 
\begin{equation}\label{PohoYoung}
\left(\sum_{j=1}^{n}\sqrt{(\alpha_j+1)(\alpha_j+2)}c_{\alpha+2e_j}\right)c_{\alpha}\geq -\epsilon_k c_{\alpha}^2-\frac{1}{4\epsilon_k}\left(\sum_{j=1}^{n}\sqrt{(\alpha_j+1)(\alpha_j+2)}c_{\alpha+2e_j}\right)^2    
\end{equation}
for some $\epsilon_k>0$ to be determined. Then by the Cauchy-Schwarz inequality, we have
\begin{equation*}
\begin{aligned}
\left(\sum_{j=1}^{n}\sqrt{(\alpha_j+1)(\alpha_j+2)}c_{\alpha+2e_j}\right)^2
\leq& \left(\sum_{j=1}^{n}(\alpha_j+1)\right)\left(\sum_{j=1}^{n}(\alpha_j+2)c_{\alpha+2e_j}^2\right)\\
=&(k+n)\sum_{j=1}^{n}(\alpha_j+2)c_{\alpha+2e_j}^2.
\end{aligned}    
\end{equation*}
In addition, 
\begin{equation*}
\begin{aligned}
\sum_{k=1}^{\infty}\sum_{\alpha\in \mathbb{N}^n,|\alpha|=k}\sum_{j=1}^{n}\frac{kt_k(k+n)(\alpha_j+2)}{4\epsilon_k(\beta+\frac{k}{2})}c_{\alpha+2e_j}^2
\leq& \sum_{k=1}^{\infty}\sum_{\alpha\in \mathbb{N}^n,|\alpha|=k+2}\sum_{j=1}^{n}\frac{kt_k(k+n)\alpha_j}{4\epsilon_k(\beta+\frac{k}{2})}c_{\alpha}^2\\
=&\sum_{k=1}^{\infty}\sum_{\alpha\in \mathbb{N}^n,|\alpha|=k+2}\frac{kt_k(k+n)(k+2)}{4\epsilon_k(\beta+\frac{k}{2})}c_{\alpha}^2\\
=&\sum_{k=3}^{\infty}\sum_{\alpha\in \mathbb{N}^n,|\alpha|=k}\frac{(k-2)t_{k-2}(k+n-2)k}{4\epsilon_{k-2}(\beta+\frac{k}{2}-1)}c_{\alpha}^2.
\end{aligned}
\end{equation*}
It follows that
\begin{equation*}
\begin{aligned}
&\sum_{k=1}^{\infty}\sum_{\alpha\in \mathbb{N}^n,|\alpha|=k}\frac{kt_k}{4\epsilon_k(\beta+\frac{k}{2})}\left(\sum_{j=1}^{n}\sqrt{(\alpha_j+1)(\alpha_j+2)}c_{\alpha+2e_j}\right)^2\\
\leq &\sum_{k=3}^{\infty}\sum_{\alpha\in \mathbb{N}^n,|\alpha|=k}\frac{(k-2)(k+n-2)k}{4\epsilon_{k-2}(\beta+\frac{k}{2}-1)}t_{k-2}c_{\alpha}^2. 
\end{aligned}
\end{equation*}
Recall that  $t_k=\frac{\Gamma(\beta+\sigma+\frac{k}{2})}{\Gamma(\beta+\frac{k}{2})}$, we have for each $k\geq 3$, $t_{k-2}=\frac{\beta+\frac{k}{2}-1}{\beta+\sigma+\frac{k}{2}-1}t_k$. Hence
\begin{equation}\label{Pohomix}
\begin{aligned}
&\sum_{k=1}^{\infty}\sum_{\alpha\in \mathbb{N}^n,|\alpha|=k}\frac{kt_k}{4\epsilon_k(\beta+\frac{k}{2})}\left(\sum_{j=1}^{n}\sqrt{(\alpha_j+1)(\alpha_j+2)}c_{\alpha+2e_j}\right)^2 \\
\leq& \sum_{k=3}^{\infty}\sum_{\alpha\in \mathbb{N}^n,|\alpha|=k}\frac{(k-2)(k+n-2)k}{4\epsilon_{k-2}(\beta+\sigma+\frac{k}{2}-1)}t_{k}c_{\alpha}^2. 
\end{aligned}
\end{equation}
Combining \eqref{Poho}, \eqref{PohoYoung} and \eqref{Pohomix}, we obtain that
\begin{equation}\label{3.10}
\sum_{k=1}^{\infty}\sum_{\alpha\in \mathbb{N}^n,|\alpha|=k}\left(2\beta+k-\epsilon_k-\frac{(k-2)(k+n-2)(\beta+\frac{k}{2})}{4\epsilon_{k-2}(\beta+\sigma+\frac{k}{2}-1)}\bm{1}_{k\geq 3}\right)\frac{kt_k}{\beta+\frac{k}{2}}c_{\alpha}^2\leq 0,
\end{equation}
where $$
\bm{1}_{k\geq 3}=\left\{\begin{array}{lll}
0,&\quad\text{if}~k=1, 2,\\
1,&\quad\text{if}~k\geq 3.
\end{array}
\right.$$

In order to finish the proof of Theorem \ref{main}, it suffices to choose $\epsilon_k>0$ so that each coefficient in the summation above is positive. 

We will divide the discussions into two cases.

\textbf{Case 1: $n\leq 2\sigma$ or $1<p<\frac{n+2\sigma}{n-2\sigma}$.}

In this case, we choose $\epsilon_1=\frac{n}{2}+1-\sigma$, $\epsilon_2=\frac{n}{2}+2-\sigma$ and
\begin{equation*}
\epsilon_k=\frac{n}{2}+k-\sigma-\frac{(k-2)(k+n-2)(\beta+\frac{k}{2})}{4\epsilon_{k-2}(\beta+\sigma+\frac{k}{2}-1)}, \ k\geq 3.
\end{equation*}
If $\epsilon_{k}>0$ for all $k\geq 1$, then \eqref{3.10} yields
\begin{equation*}
\sum_{k=1}^{\infty}\sum_{\alpha\in \mathbb{N}^n,|\alpha|=k}(2\beta+\sigma-\frac{n}{2})\frac{kt_k}{\beta+\frac{k}{2}}c_{\alpha}^2\leq 0.
\end{equation*}
Since $n\leq2\sigma$ or $1<p<\frac{n+2\sigma}{n-2\sigma}$, we have $\beta:=\frac{\sigma}{p-1}>\frac{n-2\sigma}{4}$ and we conclude that $c_{\alpha}=0$ for all $|\alpha|\geq 1$.

It remains to check that $\epsilon_k>0$ for each $k\geq 1$. In fact, we will prove inductively that 
\begin{equation*}
\epsilon_k\geq \frac{k}{2}+1-\sigma    
\end{equation*}
for each $k\geq 1$.

Firstly, we can check that it holds for $\epsilon_1$ and $\epsilon_2$. Now suppose $\epsilon_{k-2}\geq \frac{k-2}{2}+1-\sigma$ for some $k\geq 3$, then we have
\begin{equation*}
\epsilon_k\geq \frac{n}{2}+k-\sigma-\frac{(k-2)(k+n-2)(\beta+\frac{k}{2})}{2(k-2\sigma)(\beta+\sigma+\frac{k}{2}-1)}.
\end{equation*}
It suffices to prove
\begin{equation*}
\frac{(k-2)(k+n-2)(\beta+\frac{k}{2})}{2(k-2\sigma)(\beta+\sigma+\frac{k}{2}-1)}\leq \frac{n+k-2}{2}.
\end{equation*}
This is equivalent to 
$$\frac{(k-2)(\beta+\frac{k}{2})}{(k-2\sigma)(\beta+\sigma+\frac{k}{2}-1)}\leq 1$$
or 
\begin{equation}\label{sigma restriction1}
2(1-\sigma)(\beta+\sigma)\geq 0.
\end{equation}
It holds automatically since  $0<\sigma<1$.

Consequently, we complete the proof for this case.

\vspace{3mm}

\textbf{Case 2: $n>2\sigma$ and $p=\frac{n+2\sigma}{n-2\sigma}$.}

In this case, we choose $\epsilon_k=2\beta+\frac{k}{2}=\frac{n+k}{2}-\sigma>0$.

It suffices to prove that each coefficient
\begin{equation*}
\gamma_k:=2\beta+k-\epsilon_k-\frac{(k-2)(k+n-2)(\beta+\frac{k}{2})}{4\epsilon_{k-2}(\beta+\sigma+\frac{k}{2}-1)}\bm{1}_{k\geq 3}>0
\end{equation*}
for all $k\geq 1$.

For $k=1,2$, we have $\gamma_k=\frac{k}{2}>0$. For $k\geq 3$, 
\begin{equation*}
\gamma_k=\frac{k}{2}-\frac{(k-2)(k+n-2)(\beta+\frac{k}{2})}{2(n+k-2-2\sigma)(\beta+\sigma+\frac{k}{2}-1)}.
\end{equation*}
Then $\gamma_{k}>0$ is equivalent to 
\begin{equation}\label{sigmarestriction2}
2\beta(k(1-\sigma)+n-2)+k\sigma(n-2\sigma)>0.
\end{equation}

If $n\geq 2$, then it holds automatically.

If $n=1$, then we have $0<\sigma<\frac{1}{2}$ and
\begin{equation*}
2\beta(k(1-\sigma)+n-2)+k\sigma(n-2\sigma)>2\beta(\frac{k}{2}+n-2)>0    
\end{equation*}
for all $k\geq 3$. Hence we complete the proof for this case.

Combining the two cases above, we finish the proof of Theorem \ref{main}.
\end{proof}
\begin{remark}
The above arguments can also give a spectral proof of \cite[Theorem 1]{Giga-Kohn1985} by taking $\sigma=1$. Indeed, as mentioned in the proof of Proposition \ref{GK1}, \eqref{classical Pohozaev} is exactly Giga-Kohn's Pohozaev identity \eqref{Giga-Kohn Pohozaev identity} when $\sigma=1$. In the proof of Theorem \ref{main}, the assumption that $0<\sigma<1$ is only used in \eqref{sigma restriction1} and \eqref{sigmarestriction2}. When $\sigma=1$, these two conditions still hold. So we can also choose $\epsilon_{k}$ so that the coefficients in \eqref{3.10} are positive.
\end{remark}

\section{Monotonicity formula}\label{Sec:Monotonicity}
In this section, we establish a monotonicity formula for the energy
\begin{equation*}
\begin{aligned}
\mathcal E[v](s)
=&\frac{1}{4}\int_{\R^n}\int_{\R^n} (v(z,s)-v(w,s))^2K(z,w)dwdz\\
&+\frac{1}{2}\int_{-\infty}^{0}B_\tau[v(\cdot,s)-v(\cdot,s+\tau),v(\cdot,s)-v(\cdot,s+\tau)]d\tau\\
&+\frac{A}{2}\int_{\R^n}\rho v^2-\frac{1}{p+1}\int_{\R^n}\rho |v|^{p+1},   
\end{aligned}
\end{equation*}
where the symmetric bilinear form $B_\tau$ is defined as
\begin{equation}
B_\tau[f,g]:=\int_{\R^n}\int_{\R^n}f(z)g(w)M_\tau(z,w)\rho(z)dwdz    \end{equation}
for any $f,g\in L^2_\rho(\mathbb{R}^n)$ and $\tau<0$. We also define the  bilinear form $Q_\tau$ 
\begin{equation}
Q_\tau[f,g]:=\int_{\R^n}\int_{\R^n}f(z)g(w)\frac{\partial M_{\tau}}{\partial \tau}(z,w)\rho(z)dwdz
\end{equation}
for any $f,g\in L^2_\rho(\mathbb{R}^n)$ and $\tau<0$.

First, we show that the bilinear forms $B_\tau$ and $Q_\tau$ are both well-defined, bounded and positive-definite.
\begin{lemma}\label{Quadform}
For each fixed $\tau<0$, we have
\begin{equation*}
|B_\tau[f,g]|\leq  m(\tau)\|f\|_{L^2_\rho(\mathbb{R}^n)}\|g\|_{L^2_\rho(\mathbb{R}^n)},
\end{equation*}
and
\begin{equation*}
|Q_\tau[f,g]|\leq C\left(m'(\tau)+(1-e^\tau)^{-1}m(\tau)\right)\|f\|_{L^2_\rho(\mathbb{R}^n)}\|g\|_{L^2_\rho(\mathbb{R}^n)},
\end{equation*}
for any $f,g\in L^2_\rho(\mathbb{R}^n)$. Moreover, the quadratic forms $B_{\tau}$ and $Q_{\tau}$ are ‌positive definite.
\end{lemma}
\begin{proof}
 For each fixed $\tau<0$, we get  from Lemma \ref{CK} that
\begin{equation*}
\begin{aligned}
|B_\tau[f,g]|
&=m(\tau)\left|\int_{\R^n}\int_{\R^n}\int_{\R^n}\rho(\eta)p_{\tau/2}(\eta,z)p_{\tau/2}(\eta,w)f(z)g(w)dwdzd\eta\right|\\
&=m(\tau)\int_{\R^n}\left(\int_{\R^n}p_{\tau/2}(\eta,z)f(z)dz\right)\left(\int_{\R^n}p_{\tau/2}(\eta,z)g(z)dz\right)\rho(\eta)d\eta\\
&\leq m(\tau)\left(\int_{\R^n}\left(\int_{\R^n}p_{\tau/2}(\eta,z)f(z)dz\right)^2\rho(\eta)d\eta\right)^{\frac{1}{2}} \cdot \left(\int_{\R^n}\left(\int_{\R^n}p_{\tau/2}(\eta,z)g(z)dz\right)^2\rho(\eta)d\eta\right)^{\frac{1}{2}}.
\end{aligned}
\end{equation*}
Since 
\begin{equation*}
\int_{\mathbb{R}^n}p_\tau (\eta,z)dz=1
\end{equation*}
and
\begin{equation*}
\int_{\mathbb{R}^n}\rho(\eta)p_\tau (\eta,z)d\eta=\rho(z)
\end{equation*}
for any $\tau<0$, we deduce from Cauchy-Schwarz inequality that
\begin{equation*}
\begin{aligned}
\int_{\R^n}\left(\int_{\R^n}p_{\tau/2}(\eta,z)f(z)dz\right)^2\rho(\eta)d\eta
\leq &\int_{\R^n}\int_{\R^n}  \rho(\eta)p_{\tau/2}(\eta,z)|f(z)|^2dz d\eta \\
= & \int_{\mathbb{R}^n}|f(z)|^2 \rho(z)dz.
\end{aligned}
\end{equation*}
It follows that
\begin{equation*}
|B_\tau[f,g]|\leq m(\tau)\|f\|_{L^2_\rho(\mathbb{R}^n)}\|g\|_{L^2_\rho(\mathbb{R}^n)}.
\end{equation*}
Hence, $B_{\tau}[f,g]$ is well-defined and bounded for any $f,g\in L^2_\rho(\mathbb{R}^{n})$.

We also have
\begin{equation*}
\begin{aligned}
B_\tau[f,f]
&=m(\tau)\int_{\R^n}\int_{\R^n}\int_{\R^n}\rho(\eta)p_{\tau/2}(\eta,z)p_{\tau/2}(\eta,w)f(z)f(w)dwdzd\eta\\
&=m(\tau)\int_{\R^n}\left(\int_{\R^n}p_{\tau/2}(\eta,z)f(z)dz\right)^2\rho(\eta)d\eta\geq 0.
\end{aligned}
\end{equation*}
The equality holds if and only if $f$ is identically zero.

In order to verify the properties of $Q_{\tau}$, we define
\begin{equation*}
G_\tau[f](\eta):=\int_{\R^n}p_{\tau/2}(\eta,z)f(z)dz.
\end{equation*}
It follows that
\begin{equation}\label{Qtauest}
Q_\tau[f,f]=\frac{d}{d\tau}B_\tau[f,f]=m'(\tau)\int_{\R^n}\rho |G_\tau[f]|^2 d\eta+2m(\tau)\int_{\R^n}\rho G_\tau[f] \partial_\tau G_\tau[f] d\eta.
\end{equation}
Notice that
\begin{equation*}
\begin{aligned}
\partial_\tau p_\tau(\eta,z)
&=\left[\frac{ne^\tau}{2(1-e^\tau)}-\frac{e^\tau |z-e^{\frac{\tau}{2}}\eta|^2}{4(1-e^\tau)^2}+\frac{e^{\frac{\tau}{2}}}{4(1-e^\tau)}\langle z-e^{\frac{\tau}{2}}\eta, \eta\rangle\right]p_\tau(\eta,z)\\
&=-\Delta_\eta p_{\tau}(\eta,z)+\frac{1}{2}\eta\cdot \nabla_\eta p_\tau (\eta,z).
\end{aligned}
\end{equation*}
Then
\begin{equation*}
\partial_\tau G_\tau[f]=-\frac{1}{2}\Delta_\eta G_{\tau}[f]+\frac{1}{4}\eta\cdot \nabla_\eta G_\tau[f]. 
\end{equation*}
Plugging it into \eqref{Qtauest}, we have
\begin{equation}\label{QtauG}
\begin{aligned}
Q_\tau[f,f]
&=m'(\tau)\int_{\R^n}\rho |G_\tau[f]|^2 d\eta-m(\tau)\int_{\R^n}\rho G_\tau[f] \Delta_\eta G_\tau[f] d\eta\\
&~~~~+\frac{1}{2}m(\tau)\int_{\R^n}\rho G_\tau[f] \eta \cdot \nabla_\eta G_\tau[f] d\eta\\
&=m'(\tau)\int_{\R^n}\rho |G_\tau[f]|^2 d\eta+m(\tau)\int_{\R^n}\rho |\nabla_\eta G_\tau[f]|^2 d\eta.
\end{aligned}
\end{equation}

\textbf{Claim:} For any $f\in L^2_\rho(\mathbb{R}^n)$, we have
\begin{equation*}
\int_{\R^n}\rho |\nabla_\eta G_\tau[f]|^2 d\eta\leq C(1-e^\tau)^{-1}\|f\|_{L^2_\rho(\mathbb{R}^n)}^2   
\end{equation*}

Indeed, it follows from Cauchy–Schwarz inequality that
\begin{equation*}
\begin{aligned}
|\nabla_\eta G_\tau[f](\eta)|^2
=&\left|\int_{\mathbb{R}^n}\frac{e^{\frac{\tau}{4}}}{2(1-e^{\frac{\tau}{2}})}(z-e^{\frac{\tau}{4}}\eta)p_{\frac{\tau}{2}}(\eta,z)f(z) dz\right|^2\\
\leq &\frac{Ce^{\frac{\tau}{2}}}{(1-e^{\frac{\tau}{2}})^2}\left(\int_{\mathbb{R}^n}|z-e^{\frac{\tau}{4}}\eta|^2p_{\frac{\tau}{2}}(\eta,z)dz\right)\left(\int_{\mathbb{R}^n}p_{\frac{\tau}{2}}(\eta,z)|f(z)|^2dz\right)\\
\leq &C(1-e^{\frac{\tau}{2}})^{-1}\int_{\mathbb{R}^n}p_{\frac{\tau}{2}}(\eta,z)|f(z)|^2dz.
\end{aligned}
\end{equation*}

Then
\begin{equation*}
\begin{aligned}
\int_{\R^n}\rho |\nabla_\eta G_\tau[f]|^2 d\eta
\leq& C (1-e^{\frac{\tau}{2}})^{-1}\int_{\mathbb{R}^n}\left(\int_{\mathbb{R}^n}p_{\frac{\tau}{2}}(\eta,z)\rho(\eta)d\eta\right)|f(z)|^2dz\\
\leq &C (1-e^\tau)^{-1}\|f\|_{L^2_\rho(\mathbb{R}^n)}^2.
\end{aligned}
\end{equation*}

Consequently, by \eqref{QtauG}, we have
\begin{equation*}
|Q_\tau[f,f]|\leq C\left(m'(\tau)+(1-e^\tau)^{-1}m(\tau)\right)\|f\|_{L^2_\rho(\mathbb{R}^n)}^2.
\end{equation*}
For any $f,g\in L^2_\rho(\mathbb{R}^{n})$, we can rewrite $Q_\tau[f,g]$ in terms of a bilinear form associated with \eqref{QtauG} and get
\begin{equation*}
|Q_{\tau}[f,g]|\leq |Q_{\tau}[f,f]|^{\frac{1}{2}}|Q_{\tau}[g,g]|^{\frac{1}{2}}\leq C\left(m'(\tau)+(1-e^\tau)^{-1}m(\tau)\right)\|f\|_{L^2_\rho(\mathbb{R}^n)}\|g\|_{L^2_\rho(\mathbb{R}^n)}.
\end{equation*}
Hence, $Q_{\tau}[f,g]$ is well-defined and bounded for any $f,g\in L^2_\rho(\mathbb{R}^{n})$.

Since $m(\tau), m'(\tau)> 0$ for every $\tau<0$, by \eqref{QtauG}, we have $Q_\tau[f,f]\geq 0$ and the equality holds if and only if $f$ is identically zero.
\end{proof}

\begin{proposition}[Monotonicity formula]\label{monotonicity}
Suppose $u$ is a weak solution to \eqref{fraceqn} satisfying \eqref{type I assumption}. Let $v$ be defined by \eqref{self-similar coordinates} and \eqref{self-similar transform}. Then 
\begin{equation}
\begin{aligned}
&\mathcal{E}[v](a)-\mathcal{E}[v](b)=\ \frac12\int_a^b \int_{-\infty}^{0}Q_\tau[v(\cdot,s)-v(\cdot,s+\tau),v(\cdot,s)-v(\cdot,s+\tau)]d\tau ds\geq 0
\end{aligned}
\end{equation}
for any real number $a<b$. Moreover, if $\mathcal{E}[v](a)$ is a constant function with respect to $a$ on $\mathbb{R}$, then $v$ is independent of $s$.
\end{proposition}
\begin{proof}
For notational simplicity, let us define $\delta_\tau f(z,s):=f(z,s)-f(z,s+\tau)$ for any function $f$. In view of Lemma \ref{FracOperatorSelfSim}, $v$ satisfies
\begin{equation*}
\begin{aligned}
&\int_{-\infty}^{0}\int_{\R^n}(v(z,s)-v(w,s+\tau))M_\tau(z,w)dwd\tau+Av(z,s)=|v(z,s)|^{p-1}v(z,s).
\end{aligned}
\end{equation*}

Fix $a<b$ and pick $0<h<1/10$. Testing the equation above against $h^{-1}\delta_h v(z,s)\rho(z)$ and integrating in $z$ and $s$ \footnote{If $p+2\sigma>2$, we have $v$ is differentiable in $s$. Then we can test the equation against $v_s$ to simplify the computations.}, we have
\begin{equation}\label{mono1}
\begin{aligned}
0=&\int_a^b\int_{-\infty}^{0}\int_{\mathbb{R}^n}\int_{\mathbb{R}^n}(v(z,s)-v(w,s+\tau))h^{-1}\delta_h v(z,s)M_\tau(z,w)\rho(z)dwdzd\tau ds\\
&+A\int_a^b\int_{\R^n}\rho vh^{-1}\delta_h vdzds-\int_a^b\int_{\R^n}\rho |v|^{p-1}vh^{-1}\delta_h vdzds.
\end{aligned}
\end{equation}

For the first term, we can split it into
\begin{equation*}
\begin{aligned}
&\int_a^b\int_{-\infty}^{0}\int_{\mathbb{R}^n}\int_{\mathbb{R}^n}(v(z,s)-v(w,s+\tau))h^{-1}\delta_h v(z,s)M_\tau(z,w)\rho(z)dwdzd\tau ds\\
=&\int_a^b\int_{-\infty}^{0}\int_{\mathbb{R}^n}\int_{\mathbb{R}^n}(v(z,s)-v(w,s))h^{-1}\delta_h v(z,s)M_\tau(z,w)\rho(z)dwdzd\tau ds\\
&+\int_a^b\int_{-\infty}^{0}B_\tau[\delta_\tau v(\cdot,s),h^{-1}\delta_h v(\cdot,s)]d\tau ds\\
=&: I_1+I_2.
\end{aligned}
\end{equation*}

For the term $I_1$, we use the symmetry of $M_\tau(z,w)\rho(z)$ to get
\begin{equation*}
\begin{aligned}
I_1=&\frac{1}{2h}\int_a^b\int_{-\infty}^{0}\int_{\mathbb{R}^n}\int_{\mathbb{R}^n}(v(z,s)-v(w,s))(\delta_h v(z,s)-\delta_h v(w,s))M_\tau(z,w)\rho(z)dwdzd\tau ds\\
=&\frac{1}{4}\int_a^b\int_{-\infty}^{0}\int_{\mathbb{R}^n}\int_{\mathbb{R}^n}\frac{\delta_h(v(z,s)-v(w,s))^2}{h}M_\tau(z,w)\rho(z)dwdzd\tau ds\\
&+\frac{1}{4h}\int_a^b\int_{-\infty}^{0}\int_{\mathbb{R}^n}\int_{\mathbb{R}^n}(\delta_h v(z,s)-\delta_h v(w,s))^2M_\tau(z,w)\rho(z)dwdzd\tau ds\\
=&:I_{11}+I_{12}.
\end{aligned}
\end{equation*}
For $I_{11}$, we have
\begin{equation*}
\begin{aligned}
\lim_{h\to 0}I_{11}=&\frac{1}{4}\int_{-\infty}^{0}\int_{\mathbb{R}^n}\int_{\mathbb{R}^n}(v(z,a)-v(w,a))^2M_\tau(z,w)\rho(z)dwdzd\tau\\
&-\frac{1}{4}\int_{-\infty}^{0}\int_{\mathbb{R}^n}\int_{\mathbb{R}^n}(v(z,b)-v(w,b))^2M_\tau(z,w)\rho(z)dwdzd\tau\\
=&\frac{1}{4}\int_{\mathbb{R}^n}\int_{\mathbb{R}^n}(v(z,a)-v(w,a))^2K(z,w)dwdz\\
&-\frac{1}{4}\int_{\mathbb{R}^n}\int_{\mathbb{R}^n}(v(z,b)-v(w,b))^2K(z,w)dwdz.
\end{aligned}
\end{equation*}
For $I_{12}$, by Proposition \ref{aprioriv} with some $\frac{1+\sigma}{2}<\theta<\min\{\frac{p}{2}+\sigma,1\}$, we have both
\begin{equation*}
\begin{aligned}
|\delta_h v(z,s)-\delta_h v(w,s)|
\leq& |v(z,s)-v(w,s)-\nabla_{z} v(z,s)\cdot (z-w)|\\
&+|v(z,s+h)-v(w,s+h)-\nabla_{z}v(z,s+h)\cdot (z-w)|\\
&+|\nabla_{z}v(z,s+h)-\nabla_{z}v(z,s)|\cdot |z-w|\\
\leq &C\max\{|z-w|^{2\theta}, |z-w|^{2}\}+C(1+|z|)h^{\theta}\cdot|z-w|\\
\end{aligned}
\end{equation*}
and
\begin{equation*}
|\delta_h v(z,s)-\delta_h v(w,s)|\leq |\delta_h v(z,s)|+|\delta_h v(w,s)|\leq C(1+|z|+|w|)h^{\theta}.
\end{equation*}
It follows that
\begin{equation*}
|\delta_h v(z,s)-\delta_h v(w,s)|\leq C(1+|z|+|w|)h^{\theta-\frac{1}{2}}\min\{h^{\frac{1}{2}},|z-w|\}.    
\end{equation*}
Plugging it into $I_{12}$, we have
\begin{equation*}
|I_{12}|\leq C\int_a^b\int_{-\infty}^{0}\int_{\mathbb{R}^n}\int_{\mathbb{R}^n}(1+|z|+|w|)^2h^{2\theta-2}\min\{h,|z-w|^2\}M_\tau(z,w)\rho(z)dwdzd\tau ds.
\end{equation*}
One can check that
\begin{equation}\label{The first estimate}
\int_{\mathbb{R}^n}\int_{\mathbb{R}^n}(1+|z|+|w|)^2p_\tau(z,w)\rho(z)dwdz\leq C
\end{equation}
and
\begin{equation}\label{The first estimate'}
\int_{\mathbb{R}^n}\int_{\mathbb{R}^n}(1+|z|+|w|)^2|z-w|^2p_\tau(z,w)\rho(z)dwdz\leq C(1-e^\tau).
\end{equation}
Thus
\begin{equation*}
\begin{aligned}
|I_{12}|
\leq & Ch^{2\theta-2}\int_a^b\int_{-\infty}^{0}m(\tau)\min\{h,1-e^\tau\}d\tau ds\\
\leq &Ch^{2\theta-1}\int_a^b\int_{-\infty}^{\log(1-h)}\frac{1}{(1-e^{\tau})^{1+\sigma}}d\tau ds+Ch^{2\theta-2}\int_a^b\int_{\log(1-h)}^{0}\frac{1}{(1-e^{\tau})^{\sigma}}d\tau ds\\
\leq & Ch^{2\theta-1-\sigma}.
\end{aligned}
\end{equation*}
It follows that $|I_{12}|\to 0$ as $h\to 0^+$. 

In order to deal with the term $I_2$, we  split it as
\begin{equation*}
\begin{aligned}
I_2=&\frac{1}{h}\int_a^b\int_{-\infty}^{-h}B_\tau[\delta_\tau v(\cdot,s),\delta_h v(\cdot,s)]d\tau ds+\frac{1}{h}\int_a^b\int_{-h}^{0}B_\tau[\delta_\tau v(\cdot,s),\delta_h v(\cdot,s+\tau)]d\tau ds\\
=&:I_{21}+I_{22}.
\end{aligned}
\end{equation*}

 For $I_{21}$, we can further split it as
\begin{equation*}
\begin{aligned}
I_{21}=&\frac{1}{h}\int_a^b\int_{-\infty}^{-h}B_\tau[\delta_\tau v(\cdot,s),\delta_\tau\delta_h v(\cdot,s)]d\tau ds+\frac{1}{h}\int_a^b\int_{-\infty}^{-h}B_\tau[\delta_\tau v(\cdot,s),\delta_h v(\cdot,s+\tau)]d\tau ds\\
=&\frac{1}{h}\int_a^b\int_{-\infty}^{-h}B_\tau[\delta_\tau v(\cdot,s),\delta_h\delta_\tau v(\cdot,s)]d\tau ds+\frac{1}{h}\int_a^b\int_{-\infty}^{-h}B_\tau[\delta_\tau v(\cdot,s),\delta_h v(\cdot,s+\tau)]d\tau ds\\
=&:I_{211}+I_{212}.
\end{aligned}
\end{equation*}

For $I_{211}$, we use the symmetry of $B_\tau$ to obtain
\begin{equation*}
\begin{aligned}
I_{211}=&\frac{1}{2h}\int_a^b\int_{-\infty}^{-h}B_\tau[\delta_\tau v(\cdot,s),\delta_\tau v(\cdot,s)]-B_\tau[\delta_\tau v(\cdot,s+h),\delta_\tau v(\cdot,s+h)]d\tau ds\\
&+\frac{1}{2h}\int_a^b\int_{-\infty}^{-h}B_\tau[\delta_h\delta_\tau v(\cdot,s),\delta_h\delta_\tau v(\cdot,s)]d\tau ds\\
=&:I_{2111}+I_{2112}.
\end{aligned}
\end{equation*}
For $I_{2111}$,  we have
\begin{equation*}
\begin{aligned}
\lim_{h\to 0}I_{2111}=\frac{1}{2}\int_{-\infty}^{0}B_\tau[\delta_\tau v(\cdot,a),\delta_\tau v(\cdot,a)]-B_\tau[\delta_\tau v(\cdot,b),\delta_\tau v(\cdot,b)]d\tau.
\end{aligned}
\end{equation*}
By Proposition \ref{aprioriv}, we have
\begin{equation*}
|\delta_h\delta_\tau v(z,s)|\leq C(1+|z|)\min\{h^\theta, |\tau|^{\theta}\}.
\end{equation*}
Then \eqref{The first estimate} yields
\begin{equation*}
|I_{2112}|\leq Ch^{2\theta-1}\int_{-\infty}^{-h}m(\tau)d\tau\leq Ch^{2\theta-1}\int_{1-e^{-h}}^{1}\xi^{-1-\sigma}(1-\xi)^{\beta+\sigma-1}d\xi\leq Ch^{2\theta-1-\sigma}.
\end{equation*}
Therefore, $|I_{2112}|\to 0$ as $h\to 0$.

For $I_{212}$, we apply the symmetry of $B_\tau$ to get
\begin{equation*}
\begin{aligned}
I_{212}=&\frac{1}{h}\int_a^b\int_{-\infty}^{-h}B_\tau[\delta_\tau v(\cdot,s),\delta_h v(\cdot,s+\tau)]d\tau ds\\
=&\frac{1}{h}\int_a^b\int_{-\infty}^{-h}B_\tau[\delta_\tau v(\cdot,s),\delta_{\tau+h} v(\cdot,s)-\delta_{\tau} v(\cdot,s)]d\tau ds\\
=&\frac{1}{2h}\int_a^b\int_{-\infty}^{-h}B_\tau[\delta_{\tau+h} v(\cdot,s),\delta_{\tau+h} v(\cdot,s)]-B_\tau[\delta_{\tau} v(\cdot,s),\delta_{\tau} v(\cdot,s)]d\tau ds\\
&+\frac{1}{2h}\int_a^b\int_{-\infty}^{-h}B_\tau[\delta_{\tau} v(\cdot,s)-\delta_{\tau+h} v(\cdot,s),\delta_{\tau+h} v(\cdot,s)]d\tau ds\\
&+\frac{1}{2h}\int_a^b\int_{-\infty}^{-h}B_\tau[\delta_{\tau} v(\cdot,s),\delta_{\tau+h} v(\cdot,s)-\delta_{\tau} v(\cdot,s)]d\tau ds\\
=&\frac{1}{2h}\int_a^b\int_{-\infty}^{-h}B_\tau[\delta_{\tau+h} v(\cdot,s),\delta_{\tau+h} v(\cdot,s)]-B_\tau[\delta_{\tau} v(\cdot,s),\delta_{\tau} v(\cdot,s)]d\tau ds\\
&-\frac{1}{2h}\int_a^b\int_{-\infty}^{-h}B_\tau[\delta_{\tau} v(\cdot,s)-\delta_{\tau+h} v(\cdot,s),\delta_{\tau} v(\cdot,s)-\delta_{\tau+h} v(\cdot,s)]d\tau ds\\
=&:I_{2121}+I_{2122}.
\end{aligned}
\end{equation*}
For any $a<s<b$, we know that
\begin{equation}\label{The second key identity}
\begin{aligned}
&\frac{1}{2h}\int_{-\infty}^{-h}B_\tau[\delta_{\tau+h} v(\cdot,s),\delta_{\tau+h} v(\cdot,s)]-B_\tau[\delta_{\tau} v(\cdot,s),\delta_{\tau} v(\cdot,s)]d\tau \\
=&\frac{1}{2h}\int_{-\infty}^{-h}\int_{\mathbb{R}^n}\int_{\mathbb{R}^n}\bigg(\delta_{\tau+h} v(z,s)\delta_{\tau+h} v(w,s)-\delta_{\tau} v(z,s)\delta_{\tau} v(w,s)\bigg)M_{\tau}(z,w)\rho(z)dwdz d\tau \\
=&\frac{1}{2h}\int_{-\infty}^{-h}\int_{\mathbb{R}^n}\int_{\mathbb{R}^n}\delta_{\tau} v(z,s)\delta_{\tau} v(w,s)\bigg(M_{\tau-h}(z,w)-M_{\tau}(z,w)\bigg)\rho(z)dwdz d\tau \\
&+\frac{1}{2h}\int_{-h}^{0}\int_{\mathbb{R}^n}\int_{\mathbb{R}^n}\delta_{\tau} v(z,s)\delta_{\tau} v(w,s)M_{\tau-h}(z,w)\rho(z)dwdz d\tau.
\end{aligned}
\end{equation}
If $\tau\in (-h, 0)$, we get from Proposition \ref{aprioriv} that
\begin{equation*}
|\delta_{\tau}v(z,s)\delta_{\tau}v(w,s)|\leq C(1+|z|)(1+|w|)|\tau|^{2\theta}.
\end{equation*}
Then
\begin{equation}\label{The third estimate}
\begin{aligned}
&\frac{1}{2h}\int_{-h}^{0}\int_{\mathbb{R}^n}\int_{\mathbb{R}^n}\delta_{\tau} v(z,s)\delta_{\tau} v(w,s)M_{\tau-h}(z,w)\rho(z)dwdz d\tau\\
\leq& Ch^{2\theta-1}\int_{-h}^{0}\frac{e^{(\beta+\sigma)(\tau-h)}}{(1-e^{\tau-h})^{1+\sigma}}d\tau\leq Ch^{2\theta-1-\sigma}.
\end{aligned}
\end{equation}
Combining \eqref{The second key identity}  and \eqref{The third estimate}, we conclude that
$$
\begin{aligned}
\lim_{h\to 0}I_{2121}=&-\frac{1}{2}\int_{a}^{b}\int_{-\infty}^{0}\int_{\mathbb{R}^n}\int_{\mathbb{R}^n}\delta_{\tau} v(z,s)\delta_{\tau} v(w,s)\frac{\partial M_{\tau}}{\partial \tau}(z,w)\rho(z)dwdz d\tau\\
=&-\frac{1}{2}\int_a^b\int_{-\infty}^{0}Q_\tau[\delta_\tau v(\cdot, s),\delta_\tau v(\cdot, s)]ds.
\end{aligned}$$
Similarly, we  have
\begin{equation*}
\begin{aligned}
|I_{2122}|
\leq &\frac{C}{h}\int_a^b\int_{-\infty}^{-h}m(\tau)h^{2\theta}d\tau ds\leq Ch^{2\theta-1-\sigma}.
\end{aligned}
\end{equation*}
Therefore, we deduce that $|I_{2122}|\to 0$ as $h\to 0^+$.

For $I_{22}$, we get from Proposition \ref{aprioriv} that
\begin{equation*}
|\delta_\tau v(z,s)\delta_h v(w,s)|\leq C(1+|z|)(1+|w|)|\tau|^{\theta}h^{\theta}.
\end{equation*}
Taking this into \eqref{The first estimate}, we have
\begin{equation*}
\begin{aligned}
|I_{22}|\leq Ch^{\theta-1}\int_a^b\int_{-h}^{0}m(\tau)|\tau|^{\theta}d\tau
\leq Ch^{2\theta-1-\sigma}.
\end{aligned}
\end{equation*}
 Hence, we deduce that $\lim_{h\to 0}I_{22}=0$.

Combining the estimates from $I_{11}$ to $I_{22}$, we obtain the desired result.
\end{proof}

\section{Liouville Theorem}
In this section, we will give the proof of Theorem \ref{Liouville} by combining Theorem \ref{main} and the monotonicity formula derived in the previous section. The argument here is essentially the same as the proof of \cite[Theorem 2']{Giga-Kohn1985}.

Recall that we suppose $u$ is a weak solution of \eqref{fraceqn} satisfying \eqref{type I assumption}
and
\begin{equation*}
    \limsup_{t\to 0^-}(-t)^{\beta}|u(0,t)|>0.
\end{equation*}
Using self-similar coordinates introduced in \eqref{self-similar coordinates} and \eqref{self-similar transform}, Theorem \ref{Liouville} is equivalent to the following Liouville theorem.
\begin{theorem}\label{selfsimLiouville}
 For $n\leq 2\sigma$ or $1<p\leq \frac{n+2\sigma}{n-2\sigma}$, if $u$ is a weak solution to \eqref{fraceqn} satisfying \eqref{type I assumption}. Let $v$ be defined via \eqref{self-similar coordinates} and \eqref{self-similar transform}. If we further assume  
\begin{equation*}
    \limsup_{s\to \infty}|v(0,s)|>0,
\end{equation*}
then $v=\pm \kappa$.
\end{theorem}

Since $u$ satisfies \eqref{type I assumption}, the rescaled solutions
\begin{equation*}
u_{\lambda}(x,t):=\lambda^{2\beta}u(\lambda x,\lambda^2 t)  
\end{equation*}
remain bounded independent of $\lambda$ away from $t=0$.

Let $\lambda_j\to 0$ and  $\lambda_j'\to \infty$. From Proposition \ref{aprioriu}, we can (up to some subsequences) take limits
\begin{equation*}
\lim_{\lambda_j\to 0}u_{\lambda_j}=u_0,\quad \lim_{\lambda_j'\to \infty}u_{\lambda_j'}=u_\infty
\end{equation*}
for some $u_0$ and $u_\infty$. Using self-similar coordinates, they are translated into 
\begin{equation*}
\begin{aligned}
v_{+\infty}(z,s)
&=\lim_{s_j\to \infty}v(z,s+s_j)\\
v_{-\infty}(z,s)
&=\lim_{s_j'\to -\infty}v(z,s+s_j')
\end{aligned}  
\end{equation*}
with $s_j=-\ln \lambda_j$ and $s_j'=-\ln \lambda_j'$.

In the following, we will prove Theorem \ref{selfsimLiouville}. To apply the results in Section \ref{Sec:Pohozaev}, we need the following proposition showing that all possible limits $v_{\pm \infty}$ are independent of $s$.
\begin{proposition}\label{vindeps}
 Consider a monotone increasing (respectively, decreasing) sequence such that $s_j\to \pm \infty$ and $s_{j+1}-s_j\to \pm \infty$. Assume that $v_j(z,s)=v(z,s+s_j)$ converges to a limit $v_{\pm \infty}$ locally uniformly in $\mathbb{R}^n\times \mathbb{R}$. Then the limit $v_{\pm\infty}$ is independent of $s$. Moreover, the energy $\mathcal{E}[v_{\pm\infty}]$ is independent of the choice of the sequence $\{s_j\}$.
\end{proposition}
\begin{proof}
Without loss of generality, we may assume that $s_j\to +\infty$. For any $j$, we set $v_j(z,s):=v(z,s+s_j)$. 
By Proposition \ref{aprioriv}, up to a subsequence, we have  $v_j(z,s)=v(z,s+s_j)$ converges to a limit $v_{+\infty}$ locally uniformly in $\mathbb{R}^n\times \mathbb{R}$.  Moreover, it follows from Proposition \ref{aprioriv}, \eqref{The first estimate} and \eqref{The first estimate'} that $\mathcal{E}[v](a)$ is uniformly bounded with respect to $a$. By Proposition \ref{monotonicity}, we have for any $a\in \mathbb{R}$,
\begin{equation*}
\mathcal{E}[v_j](a)-\mathcal{E}[v_{j+1}](a)=\frac{1}{2}\int_{a}^{a+s_{j+1}-s_j}\int_{-\infty}^{0}Q_\tau[\delta_\tau v(\cdot,s),\delta_\tau v(\cdot,s)]d\tau ds.
\end{equation*}
Since $\mathcal{E}[v_j](a)$ is a bounded monotone nonincreasing function, we deduce that
\begin{equation*}
\mathcal{E}[v_j](a)-\mathcal{E}[v_{j+1}](a)\to 0    
\end{equation*}
as $j\to \infty$. Since $s_{j+1}-s_j\to +\infty$, by Proposition \ref{aprioriv} and Lemma \ref{Quadform}, we have for any $a<b$,
\begin{equation*}
\int_{a}^{b}\int_{-\infty}^{0}Q_\tau[\delta_\tau v_{+\infty}(\cdot,s),\delta_\tau v_{+\infty}(\cdot,s)]d\tau ds=0.
\end{equation*}
Thus $\delta_\tau v_{+\infty}(z,s)=0$ for any  $s\in \mathbb{R}$, $\tau<0$ and $z\in \mathbb{R}^n$. In particular, $v_{+\infty}$ is independent of $s$.

In order to finish the proof of the proposition, it suffices to check $\mathcal{E}[v_{+\infty}]$ is independent of the choice of the sequence $\{s_j\}$. Indeed, if there is another sequence $\{\tilde s_j\}$ with $\tilde s_j\to +\infty$ and $\tilde s_{j+1}-\tilde s_j\to +\infty$, such that $\tilde v_j(z,s):=v(z,s+\tilde s_j)$ converges to a limit $\tilde v_{+\infty}$ locally uniformly in $\mathbb{R}^n\times \mathbb{R}$,
 and
\begin{equation*}
\mathcal{E}[v_{+\infty}]\neq \mathcal{E}[\tilde v_{+\infty}].
\end{equation*}
Without loss of generality, we may assume $\mathcal{E}[v_{+\infty}]<\mathcal{E}[\tilde v_{+\infty}]$ and $s_j<\tilde s_j$. It follows that
\begin{equation*}
\mathcal{E}[v_j](0)-\mathcal{E}[\tilde v_j](0)=\mathcal{E}[v](s_j)-\mathcal{E}[v](\tilde s_j)=\frac{1}{2}\int_{s_j}^{\tilde s_j}\int_{-\infty}^{0}Q_\tau[\delta_\tau v(\cdot,s),\delta_\tau v(\cdot,s)]d\tau ds\geq 0.
\end{equation*}
Since $\mathcal{E}[v_j](0)-\mathcal{E}[\tilde v_j](0)\to \mathcal{E}[v_{+\infty}]-\mathcal{E}[\tilde v_{+\infty}]<0$, the left-hand side is negative for sufficiently large $j$, a contradiction. Consequently, $\mathcal{E}[v_{+\infty}]$ is independent of the sequence $\{s_j\}$.
\end{proof}

With Proposition \ref{vindeps}, we can prove that $\lim_{s\to\pm\infty}v(z,s)$ exists.
\begin{proposition}\label{vlim}
For $n\leq 2\sigma$ or $1<p\leq \frac{n+2\sigma}{n-2\sigma}$, if $u$ is a weak solution to \eqref{fraceqn} satisfying \eqref{type I assumption}. Let $v$ be defined via \eqref{self-similar coordinates} and \eqref{self-similar transform}. Then the limits $\lim_{s\to\pm\infty}v(z,s)$ exist and are equal to $0$ or $\pm \kappa$. Also, the convergence is locally uniform in $\mathbb{R}^n$.
\end{proposition}
\begin{proof}
Without loss of generality, we take a sequence $\{s_j\}$ with $s_j\to +\infty$. Up to a subsequence, we can further assume that $s_{j+1}-s_j\to +\infty$. By Proposition \ref{aprioriv}, up to a subsequence, we have $v(\cdot, s+s_j)$ converges to some function $v_{+\infty}(z,s)$ locally uniformly in $\mathbb{R}^n\times \mathbb{R}$.  Then Proposition \ref{vindeps} yields that $v_{+\infty}(z,s)$ is independent of $s$. By Theorem \ref{main}, we know that $v_{+\infty}$ equals $0$ or $\pm \kappa$.

In the following, we prove that $v_{+\infty}$ is independent of the sequence $\{s_j\}$.

Suppose there are two sequences $\{s_j\}$ and $\{\tilde s_j\}$ with $s_j, \tilde s_j\to +\infty$, $s_{j+1}-s_j, \tilde s_{j+1}-\tilde s_j\to +\infty$ and
\begin{equation*}
v_j(z,s):=v(z,s+s_j)\to v_{+\infty},\quad \tilde v_j(z,s):=v(z,s+\tilde s_j)\to \tilde v_{+\infty}   
\end{equation*}
locally uniformly in $\mathbb{R}^n\times \mathbb{R}$ with $v_{+\infty}\not\equiv \tilde v_{+\infty}$. By Proposition \ref{vindeps}, we have $\mathcal{E}[v_{+\infty}]=\mathcal{E}[\tilde v_{+\infty}]$. If $\mathcal{E}[v_{+\infty}]=\mathcal{E}[0]$, then $v_{+\infty}=0$ and $\mathcal{E}[\tilde v_{+\infty}]=\mathcal{E}[v_{+\infty}]=0$, hence $\tilde v_{+\infty}=0$. Since we have assumed $v_{+\infty}\not\equiv \tilde v_{+\infty}$, this is impossible. Therefore, we have $\mathcal{E}[v_{+\infty}]=\mathcal{E}[\pm \kappa]>\mathcal{E}[0]$. Then either $v_{+\infty}=\kappa, \tilde v_{+\infty}=-\kappa$ or $v_{+\infty}=-\kappa, \tilde v_{+\infty}=\kappa$. Without loss of generality, we assume the first case occurs. Since $v$ is continuous, there is a sequence $s_j'\to +\infty$ with $v(0,s_j')=0$. Then up to a subsequence, we have $v_j'(z,s):=v(z,s+s_j')$ converges to $0$
locally uniformly in $\mathbb{R}^n\times \mathbb{R}$. Then we have $\mathcal{E}[v_{+\infty}]\neq \mathcal{E}[0]$, a contradiction.
\end{proof}

Finally, we give the proof of Theorem \ref{selfsimLiouville}.
\begin{proof}[Proof of Theorem \ref{selfsimLiouville}]
Without loss of generality, we may assume 
\begin{equation*}
    \limsup_{s\to +\infty}v(0,s)>0.
\end{equation*}
 For any $S>0$, Proposition \ref{monotonicity} yields
\begin{equation*}
\mathcal{E}[v](-S)-\mathcal{E}[v](S)=\frac{1}{2}\int_{-S}^{S}\int_{-\infty}^{0}Q_\tau[\delta_\tau v(\cdot,s),\delta_\tau v(\cdot,s)]d\tau ds.
\end{equation*}
With Proposition \ref{vlim}, we can send $S\to +\infty$ to get
\begin{equation*}
\mathcal{E}[v_{-\infty}]-\mathcal{E}[v_{+\infty}]=\frac{1}{2}\int_{-\infty}^{+\infty}\int_{-\infty}^{0}Q_\tau[\delta_\tau v(\cdot,s),\delta_\tau v(\cdot,s)]d\tau ds\geq 0,
\end{equation*}
where $v_{\pm \infty}=\lim_{s\to \pm\infty}v(z,s)$. By our assumption, we have $v_{+\infty}=\kappa$. It follows that $v_{-\infty}=\pm \kappa$. Then the left-hand side is $0$, yielding that $v$ is independent of $s$. Hence, we have $v=v_{\pm\infty}=\kappa$.
\end{proof}
\appendix

\section{Proof of some auxiliary tools}\label{Appendix A}
\setcounter{equation}{0}
In this appendix, we prove some technical tools which were used in the proofs of the main theorems. We begin with the proof of Lemma \ref{CK}.
\begin{proof}[Proof of Lemma \ref{CK}] Using the notations introduced in \eqref{nonlocal kernel}, we have
\begin{equation*}
\begin{aligned}
&\int_{\R^n}\rho(\eta)p_{\tau/2}(\eta,z)p_{\tau/2}(\eta,w)d\eta\\
=&\frac{1}{(4\pi (1-e^{\tau/2}))^{n}}\int_{\R^n}\exp\left(-\frac{|z-e^{\frac{\tau}{4}}\eta|^2}{4(1-e^{\frac{\tau}{2}})}-\frac{|w-e^{\frac{\tau}{4}}\eta|^2}{4(1-e^{\frac{\tau}{2}})}-\frac{|\eta|^{2}}{4}\right)d\eta\\
=&\frac{1}{(4\pi (1-e^{\tau/2}))^{n}}\int_{\R^n}\exp\left(-\frac{1+e^{\frac{\tau}{2}}}{4(1-e^{\frac{\tau}{2}})}\left|\eta-\frac{e^{\frac{\tau}{4}}}{1+e^{\frac{\tau}{2}}}(z+w)\right|^2-\frac{|w-e^{\frac{\tau}{2}}z|^2}{4(1-e^\tau)}-\frac{|z|^2}{4}\right)d\eta\\
=&\rho(z)p_\tau(z,w).
\end{aligned}
\end{equation*}
The proof is completed.
\end{proof}
Next, we give the proof of Lemma \ref{Teigen}.
\begin{proof}[Proof of Lemma \ref{Teigen}]
Notice that $\int_{\R^n}p_\tau(z,w)dw=1$. Then $T_{\beta,\sigma}$ can be written as
\begin{equation*}
T_{\beta,\sigma}v(z)=\int_{-\infty}^{0}m(\tau)\left(v(z)-\int_{\mathbb{R}^n}p_\tau(z,w)v(w)dw\right)d\tau+Av(z).
\end{equation*}

Consider the generating function $\Phi(z,\xi)$ of $h_{\alpha}$ (see \cite{AndrewsAskeyRoy1999}) defined as 
\begin{equation}\label{generatingPhi}
\Phi(z,\xi):=\exp\left(\frac{z\cdot \xi}{\sqrt{2}}-\frac{|\xi|^2}{2}\right)=\sum_{\alpha\in \mathbb{N}^n}h_{\alpha}(z)\frac{\xi^\alpha}{\sqrt{\alpha!}}.
\end{equation}
Then we have
\begin{equation*}
\begin{aligned}
&\int_{\R^n}p_\tau(z,w)\Phi(w,\xi)dw\\
=&\ \frac{1}{(4\pi (1-e^{\tau}))^{\frac{n}{2}}}\int_{\R^n}\exp\left(-\frac{|w-e^{\frac{\tau}{2}}z|^2}{4(1-e^\tau)}+\frac{w\cdot \xi}{\sqrt{2}}-\frac12 |\xi|^2\right)dw\\
=&\ \frac{1}{(4\pi (1-e^{\tau}))^{\frac{n}{2}}}\int_{\R^n}\exp\left(-\frac{|w-e^{\frac{\tau}{2}}z-\sqrt{2}(1-e^\tau)\xi|^2}{4(1-e^\tau)}\right)dw \cdot \exp\left(\frac{e^{\frac{\tau}{2}}z\cdot \xi}{\sqrt{2}}-\frac{e^\tau}{2} |\xi|^2\right)\\
=&\ \exp\left(\frac{e^{\frac{\tau}{2}}z\cdot \xi}{\sqrt{2}}-\frac{e^\tau}{2} |\xi|^2\right)=\Phi(z,e^{\frac{\tau}{2}}\xi).
\end{aligned}
\end{equation*}
It follows that for any fixed $\xi\in \R^n$,
\begin{equation*}
T_{\beta,\sigma}\Phi(z,\xi)=\int_{-\infty}^{0}m(\tau)\left(\Phi(z,\xi)-\Phi(z,e^{\frac{\tau}{2}}\xi)\right)d\tau+A\Phi(z,\xi).
\end{equation*}
Plugging \eqref{generatingPhi} above, we have
\begin{equation*}
\sum_{\alpha\in \mathbb{N}^n}T_{\beta,\sigma}h_{\alpha}(z)\frac{\xi^\alpha}{\sqrt{\alpha!}}=\sum_{\alpha\in \mathbb{N}^n}\left[\int_{-\infty}^{0}m(\tau)(1-e^{\frac{|\alpha|}{2}\tau})d\tau+A\right]h_{\alpha}(z)\frac{\xi^\alpha}{\sqrt{\alpha!}}.
\end{equation*}
Comparing the coefficients, we get for any $\alpha\in \mathbb{N}^n$,
\begin{equation*}
T_{\beta,\sigma}h_{\alpha}(z)= \left[ \int_{-\infty}^{0}m(\tau)(1-e^{\frac{|\alpha|}{2}\tau})d\tau+A\right]h_{\alpha}(z).  
\end{equation*}
Recall that from \eqref{IntegralA},
\begin{equation*}
A=\int_{-\infty}^{0}\frac{(1-e^{\beta\tau})e^{\sigma\tau}}{|\Gamma(-\sigma)|(1-e^\tau)^{1+\sigma}}d\tau.
\end{equation*}
Then we have
\begin{equation*}
\begin{aligned}
\int_{-\infty}^{0}m(\tau)(1-e^{\frac{|\alpha|}{2}\tau})d\tau+A=\frac{1}{|\Gamma(-\sigma)|}\int_{-\infty}^{0}e^{\sigma\tau}(1-e^\tau)^{-1-\sigma}(1-e^{(\beta+\frac{|\alpha|}{2})\tau})d\tau.
\end{aligned}
\end{equation*}
With the change of variables $y=e^\tau$, we have
\begin{equation*}
\begin{aligned}
&\frac{1}{|\Gamma(-\sigma)|}\int_{-\infty}^{0}e^{\sigma\tau}(1-e^\tau)^{-1-\sigma}(1-e^{(\beta+\frac{|\alpha|}{2})\tau})d\tau\\
=&\frac{1}{|\Gamma(-\sigma)|}\int_{0}^{1}y^{\sigma-1}(1-y)^{-1-\sigma}(1-y^{\beta+\frac{|\alpha|}{2}})dy\\
=&\frac{1}{\sigma|\Gamma(-\sigma)|}\int_{0}^{1}\frac{d}{dy}\left[\left(\frac{y}{1-y}\right)^{\sigma}\right](1-y^{\beta+\frac{|\alpha|}{2}})dy\\
=&\frac{\beta+\frac{|\alpha|}{2}    }{\sigma|\Gamma(-\sigma)|}\int_{0}^{1}y^{{\beta+\frac{|\alpha|}{2}}+\sigma-1}(1-y)^{-\sigma}dy=\frac{\Gamma({\beta+\frac{|\alpha|}{2}}+\sigma)}{\Gamma({\beta+\frac{|\alpha|}{2}})},
\end{aligned}
\end{equation*}
which is our desired result.
\end{proof}

{\center{\bf Acknowledgements}.} Y. Sire is partially supported by DMS NSF grant 2154219, ``Regularity vs singularity formation in
elliptic and parabolic equations''. J.C. Wei is supported by National Key R\&D Program of China 2022YFA1005602, and Hong Kong General Research Fund ``New frontiers in singular limits of nonlinear partial differential equations''. The research of K. Wu is supported by the National Natural Science Foundation
of China (No. 12401264) and Yunnan Revitalization Talent Support Program.

\bibliography{Fullyfractionalreference}
\bibliographystyle{plain}

\end{document}